\documentclass[10pt,a4paper]{amsart}
\usepackage[pdftex]{graphicx} 
\usepackage[usenames,dvipsnames]{xcolor}
\usepackage{hhline}
\usepackage[english]{babel} 
\usepackage[centering,margin=2.5cm]{geometry}
\usepackage{upgreek}
\usepackage[font=small,labelfont=bf]{caption}
\usepackage{amsfonts}
\usepackage{tensor} 
\usepackage{lmodern}
\usepackage{verbatim}
\usepackage{amsmath}
\usepackage[all, cmtip]{xy}\usepackage{amssymb}
\usepackage{amsthm}
\usepackage{amsbsy}
\usepackage{extpfeil}
\usepackage{bbm}
\usepackage{paralist}
\usepackage[colorlinks,citecolor=magenta,pagebackref=true,urlcolor=magenta,
pdftex]{hyperref}
\usepackage{cancel}

\usepackage{soul}

\usepackage[T1]{fontenc}

\usepackage{nicefrac}
\usepackage{microtype}
\usepackage{mathrsfs}
\usepackage[nodisplayskipstretch]{setspace}
\theoremstyle{plain}
\newtheorem{thm}{Theorem}
\newtheorem*{thm*}{Theorem}

\newtheorem{prop}{Proposition}
\newtheorem{lem}{Lemma}

\newtheorem{ques}{Question}

\newtheorem*{prb*}{Problem}

\theoremstyle{definition}
\newtheorem{rem}{Remark}
\newtheorem{df}{Definition}

\theoremstyle{remark}

\newcommand{\St}{\mr{st}}
\newcommand{\Lk}{\mr{lk}}

\newcommand{\cm}[1]{}

\newcommand\mr[1]{\mathrm{#1}}

\newcommand{\R}{\mathbb{R}}
\newcommand{\Z}{\mathbb{Z}}

\title{CAT(0) metrics on contractible manifolds}
\author{Karim A.~Adiprasito}
\address{Einstein Institute for Mathematics, Hebrew University of Jerusalem, Jerusalem, 91904 Israel}
\email{adiprasito@math.huji.ac.il}

\author{Louis Funar}
\address{Institut Fourier, UMR 5582, Laboratoire de Math\'ematiques
Universit\'e Grenoble Alpes, CS 40700, 38058 Grenoble cedex 9, France}
\email{louis.funar@univ-grenoble-alpes.fr}

\date{\today}
\thanks{K.~Adiprasito was supported by an EPDI/IPDE postdoctoral fellowship, a Minerva fellowship of the Max Planck Society and NSF
Grant DMS 1128155, L. Funar was partially supported by
the ANR 2011 BS 01 020 01 ModGroup.}


\begin{document}

\begin{abstract}
We prove that an open manifold $M$ of dimension at least $5$ which admits a 
complete CAT(0) polyhedral metric is pseudo-collarable, its fundamental group at infinity is strongly perfectly semistable and has vanishing Chapman-Siebenmann obstruction $\tau_{\infty}(M)$.
Moreover, this implies that $M$ is topologically collapsible, when $n\geq 6$. 
Conversely, any finite dimensional collapsible polyhedron is PL homeomorphic to a CAT(0) cubical complex.  
\end{abstract}
\maketitle

\section{Introduction}

\vspace{0.2cm}\noindent 
\subsection{Context}
The Cartan--Hadamard theorem in Riemannian geometry can be accentuated in two parts: 
\begin{compactenum}[(1)]
\item Nonpositive sectional curvature (a local condition) together with simple connectivity implies global nonpositive curvature (i.e. Alexandrov's CAT(0) condition).
\item The Riemannian manifold in question is in particular diffeomorphic to an Euclidean space.
\end{compactenum}

With the increase of interest in Alexandrov's coarse curvature notions (motivated chiefly by the work of Burago, Perelman, Shioya, Gromov and others) it was noticed that while the first part holds quite generally for metric length spaces \cite{BBI}, the second part of the Cartan--Hadamard seemed to break down in the topological and polyhedral categories. When revitalizing the interest in CAT(0) geometry for his work on hyperbolic groups, Gromov therefore prominently asked in the eighties for other open manifolds which can be endowed with complete CAT(0) metrics.

Gromov also noticed that this question should be asked for geodesically complete metrics (an assumption we restrict to throughout), as every manifold with boundary can be given a smooth non-complete metric of curvature $<0$ (and also a metric of curvature $>0$) using the h-principle. In this setting, a CAT(0) manifold is necessarily contractible.

A first answer to this question was provided by Davis and Januszkiewicz \cite{DJ}, who proved the existence of nontrivial CAT(0) $n$-manifolds, $n\geq 5$, using Gromov's own hyperbolization construction, combined with the Cannon--Edwards criterion. Soon after, Ancel and Guilbault \cite{AG} extended the picture by showing that the interior of any compact contractible manifold of dimension $n\geq 5$ can be given a complete CAT(-1) geodesic metric.

On the other hand, already examples of CAT(0) manifolds constructed by Davis and Januszkiewicz have fundamental groups at infinity not stable, and are therefore not compactifiable, giving us two disjoint sources for CAT(0) manifolds. A complete understanding of CAT(0) manifolds remained elusive.

\subsection{CAT(0) metrics}
It is understood that throughout this paper any CAT(0) metric which we define on a given topological space has the property 
that the metric and usual topologies agree.  

Recall that a metric space $(X,d)$ is  {\em geodesic} (also called a length space or an inner metric space) if every two points 
of it can be joined by a minimizing geodesic, namely a curve whose length equals the distance between the points. 
The length of the continuous path $\gamma:[0,1]\to X$ is defined as 
\[ \sup_{r, 0=t_0 <t_1 <\cdots t_r<t_{r+1}=1}\; \; \; \;  \sum_{j=0}^{r} d(\gamma(t_j),\gamma(t_{j+1}) \]
A geodesic triangle in $(X,d)$ satisfies the CAT($\kappa$) inequality if the geodesic comparison triangle 
with sides of the same length within  the simply connected curvature $\kappa$ Riemannian surface has 
distances between pairs of boundary points larger than those between corresponding pairs of points in the initial triangle. 
Moreover, the geodesic metric space $(X,d)$ is CAT(0) if  every geodesic triangle, which for $\kappa >0$ has perimeter less than $\frac{2\pi}{\sqrt{\kappa}}$, satisfies the CAT($\kappa$) inequality.

It is well-known that a simply connected CAT(0) space is contractible. 
Conversely, we would like to know 
which contractible spaces admit CAT(0) metrics. In this paper we will only consider this question for manifolds instead of arbitrary spaces and more specifically, {\em contractible} topological manifolds $M$ of dimension $n\geq 5$.

Classical results show that every contractible manifold $M$ is triangulable, namely there exists a {\em locally finite simplicial complex} $\Delta$ homeomorphic to $M$. 
The CAT($\kappa$) metrics which we consider on $M$ are supposed to be {\em polyhedral}, namely there exists a suitable 
triangulation $\Delta$ such that every cell of $\Delta$ when equipped with the induced metric  
is isometric to the convex hull of a finite set of points in the hyperbolic or Euclidean 
space of curvature $\kappa\leq 0$ (see \cite{BH}, I.7, Def. 7.37). The {\em piecewise
 flat equilateral metric}  
 is the length metric obtained when 
simplices or cubes are Euclidean and have all their edges of the same (unit) length. 
 We shall point out that our arguments are more general, 
it suffices that the restriction of the metric to every cell of $\Delta$ be a  
piecewise smooth Riemannian metric whose curvature is bounded above and below by two constants independent on the 
cell.

\subsection{Pseudo-collarability}
The goal of this note is to give a topological characterization of CAT(0)  manifolds for dimensions $\ge 6$. The key notion was introduced by Guilbault in \cite{Guilb}:

\begin{df}
An open manifold $M$ is {\em pseudo-collarable} if it admits an exhaustion $M=\cup_{j=1}^{\infty} M_j$ by compact manifolds $M_j$ such that $M_j\subset {\rm int}(M_{j+1})$ and 
the inclusion $\partial M_{j}\hookrightarrow M-{\rm int}(M_{j})$ is a  homotopy equivalence, for every $j\geq 1$.
\end{df}

Note that there exist open contractible $n$-manifolds which are not pseudo-collarable, for every $n\geq 5$. 

\subsection{The Chapman-Siebenmann obstruction $\tau_{\infty}$ for pseudo-collarable manifolds}
The obstruction $\tau_{\infty}$ was defined by Siebenmann in \cite{Sie3} and Chapman and Siebenmann in \cite{CS}. 
\begin{df}
Let $\varepsilon(M)$ denote the end of a one ended manifold $M$ and $U(\varepsilon(M))$ be a system 
of open  neighborhoods of infinity, namely having compact complement. 
The {\em attenuation} of the Whitehead functor is: 
\[{\rm Wh}^1(\varepsilon(M))=\lim_{\leftarrow}{}^1 \; ({\rm Wh}(\pi_1(U)))_{U\in U(\varepsilon(M))}\] where  $\lim^1$ denotes the first derived limit and ${\rm Wh}(\pi_1(K))$ denotes the  Whitehead group 
of the fundamental group $\pi_1(K)$. 

Let $M$ be an open contractible pseudo-collarable manifold. 
Consider a compact manifold exhaustion $M_i\subset {\rm int}(M_{i+1})$ of $M$, 
such that  the  inclusion map $\partial M_i \to M_{i+1}-{\rm int}(M_i)$ is a homotopy equivalence. 
Then the Chapman-Siebenmann obstruction $\tau_{\infty}(M)\in {\rm Wh}^1(\varepsilon(M))$ is 
the image of the sequence $(\tau_i)$, where 
$\tau_i$ is the image of the Whitehead torsion 
$\tau(M_{i+1}-{\rm int}(M_i),\partial M_i)\in {\rm Wh}(\pi_1(M_{i+1}))$ into 
${\rm Wh}(\pi_1(M-{\rm int}(M_i)))$.  See section \ref{obstructions} for details. 
\end{df}

\subsection{Collapsibility}
Complete CAT(0) metrics are strongly convex, namely there is an unique midpoint associated to any two points of the space.
Rolfsen proved in \cite{R1,R2} that the only compact manifolds of dimension $n\leq 3$ admitting a 
strongly convex metric are homeomorphic to the ball. This implies that  complete open  CAT(0) manifolds of dimension $n\leq 3$ are homeomorphic to $\R^n$. Thus Whitehead 3-manifolds 
cannot be endowed with complete CAT(0) metrics. 

If one looks more generally upon complexes instead of manifolds 
White (see \cite{W}) proved that a 2-complex admits a strongly convex metric if and only if it is collapsible. 
As such the result cannot be extended to higher dimensions, as there exist (non rectilinear) triangulations of the 
3-cell which are not collapsible (see \cite{Bi}).

A triangulation of some $n$-manifold 
is {\em PL} if the link of every vertex is PL homeomorphic to a $(n-1)$-sphere.    
A simply connected manifold $M$ endowed with a  PL  triangulation for which 
the associated polyhedral (piecewise flat or hyperbolic)  metric is CAT(0) is homeomorphic to $\R^n$, by 
a classical theorem of Stone (\cite{Stone}).  Thus  CAT(0) polyhedral metrics on exotic manifolds are subjacent to non PL triangulations. 

Consider two polyhedra $X$ and $Y$ such that 
$X=Y\cup e$, where $e$ is a cell whose boundary 
$\partial e$ intersects  $Y$ along the complement of a single facet.
We say that $X$ {\em elementary collapses} on $Y$, or $X$ is an {\em elementary expansion} of $Y$. 
Moreover, a collapse or an expansion is a sequence of elementary collapses or expansions, respectively.

A topological manifold will be called {\em PL collapsible} in the sequel if it admits a 
triangulation which is collapsible, i.e. it collapses to a point. In the noncompact case 
this means that dually, we can obtain it from a point by infinitely many expansions.  
We emphasize that this triangulation is {\em not} required to be a PL triangulation.

Collapsibility extends readily to CW-complexes (see \cite{Cohen}). Let 
$X$ and $Y$ be CW complexes such that $X=Y\cup_{\varphi} e$, where $e$ is a cell attached to $Y$ 
by means of some attaching map $\varphi:\partial e -f\to Y$ in the complement of a codimension one cell $f$. We say again that $X$ {\em elementary collapses} on $Y$, or $X$ is an 
{\em elementary expansion} of $Y$. Moreover, a collapse of CW complexes is a sequence of elementary expansions. If the CW complex is regular, namely all attaching maps of its cells are embeddings, 
then the complex is triangulable and the two notions of collapsibility agrees. 
This also holds when $X,Y$ are polyhedra while the attaching maps are PL. 

However we can consider general topological spaces $X$ and $Y$ and arbitrary continuous 
attaching maps in the situation above. We then say that we have a {\em topological (elementary) collapse} or {\em expansion} (see \cite{BM}). 

Recall that the {\em topological mapping cylinder}  $M(f)$  of a {\em continuous} map 
$f:M\to N$ is  given by $M\times [0,1]\cup_{(x,1)\sim f(x)} N$.  Note that $M(f)$ topologically collapses on $N$ although the collapse might not be a PL collapse if $f$ were not a PL map.

It is known (see \cite{BCC})  that there exists polyhedra which are homeomorphic to balls and hence 
topologically collapsible but which are not PL collapsible. Specifically, if  $\Sigma$ is 
a PL homology $n$-sphere with nontrivial fundamental group and $M$ is the complement of a 
PL ball embedded into the join $S^p*\Sigma$ (equivalently the $p$-th iterated suspension of $\Sigma$) whose closure is disjoint from $S^p$, then 
$M$ is not PL collapsible, although for $p\geq 1$ the double suspension theorem of Cannon 
implies that $M$ is homeomorphic to a ball.

\subsection{The  main results}
\begin{thm}\label{thm:CAT1}

An open contractible $n$-manifold $M$, $n\geq 5$ which admits a CAT(0) complete polyhedral metric is pseudo-collarable, it has strongly perfectly semistable fundamental group at infinity and 
vanishing Chapman-Siebenmann obstruction $\tau_{\infty}(M)$. 
Moreover, if $n\geq 6$, then $M$ is topologically collapsible. 
\end{thm}

\begin{thm}\label{thm:CAT2}
A finite dimensional PL collapsible polyhedral complex  is PL homeomorphic to 
a cubical complex which is CAT(0) when endowed with the piecewise flat equilateral metric.  
\end{thm}

The fact that a cubical complex which is CAT(0) with respect to the equilateral 
piecewise flat metric is indeed collapsible was established in \cite{AB}.

As an immediate consequence interiors of compact contractible manifolds are homeomorphic to CAT(0) cubical complexes, when $n\geq 4$. This should be compared with \cite{AG}, where Ancel and Guilbault proved that there are CAT(-1) metrics on these manifolds.   

In essence, pseudo-collarability guarantees an exhaustion by contractible manifolds as well as a sufficiently "nice" structure at infinity (cf.\  Lemma \ref{psdc}). A simpler notion is the notion of "geometrically contractible manifolds", which abandons the structure at infinity and describes open contractible manifolds that can be exhausted by compact contractible manifolds.
 Section \ref{variations}, which explores different topological notions related to pseudo-collarability, reveals a hierarchy:
\[
\text{compactifiable} \ \subsetneq\ \text{pseudo-collarable} \ \subsetneq\ \text{geometrically contractible} \ \subsetneq\ \text{general}. 
\]

\subsection{Comments and questions}

We formulate here two questions, for which affirmative answers might bridge the gap between 
the two main theorems. 

\begin{ques}\label{plcollapse}
Consider a  compact PL $n$-manifold $W$ with boundary, $n\geq 5$, which is homeomorphic to the 
topological mapping cylinder $M(f)$ of an acyclic map $f:M\to N$ between closed PL 
homology spheres.  
Suppose that the kernel of the map induced by $f$ at fundamental groups level is the normal closure of  
a finitely generated perfect group. Then, is the pair $(W,N)$ homeomorphic to a 
polyhedron pair $(P,Q)$ such that $P$ PL collapses onto $Q$?  
\end{ques}

In other words we ask whether any topological collapse between PL manifolds as in the statement  
can be made PL if we accept to replace the initial triangulations with non PL triangulations. 
A particular case of this question asks whether a PL manifold which topologically collapses to a point is homeomorphic to a  
polyhedron which PL collapses. This is known to hold by the construction of 
an arc spine for any compact contractible manifold $M$ by Ancel and Guilbault in \cite{AG0}. 
Specifically, if $n\geq 5$ and $M$ is a compact contractible $n$-manifold, then there exists 
a map $f:\partial M\to [0,1]$ such that $M$ is homeomorphic to the topological mapping cylinder $M(f)$. One constructs first a codimension one homology sphere $\Sigma\subset \partial M$ 
providing a surjective map at fundamental group level, so that 
$\partial M-\Sigma\times (0,1)$ is the disjoint union of two acyclic manifolds $A$ and $B$. 
Then the map $f$ sends $A$ into $0$, $\Sigma\times \{t\}$ into $t$ and $B$ into $1$.  
In particular we can refine the triangulation of $M$ such that the map $f$ becomes simplicial. 
Now, we can define the simplicial mapping cylinder $C(f)$ of $f$, which is a simplicial complex  
collapsing onto $[0,1]$ and hence to a point. By a result of Cohen (see \cite{Cohen0}) 
the simplicial mapping cylinder $C(f)$ is homeomorphic to the topological mapping cylinder $M(f)$
and in particular to $M$.

A weaker version concerns the case of open manifolds and reads  as follows: 
 
\begin{ques}\label{openplcollapse}
Consider an open contractible $n$-manifold $W$, $n\geq 5$,  which is an infinite mapping cylinder, namely the union of PL manifolds with disjoint interiors, each one homeomorphic to a topological mapping cylinder of  some acyclic map between closed PL homology spheres.  Suppose that the kernels of the induced maps at fundamental groups level are normal closures of  finitely generated perfect groups. 
Then is the manifold $W$ homeomorphic to a polyhedron which is properly  PL collapsible?
\end{ques}

In the absence of the requirement that the CAT(0) metric be polyhedral we expect the following related question. 

\begin{ques}\label{thm:CATTOP} 
An open contractible $n$-manifold $W$, $n\geq 5$, admits a CAT(-1) complete length metric if and only if it is pseudo-collarable, it has   perfectly semistable fundamental group at infinity and 
the Chapman-Siebenmann obstruction $\tau_{\infty}$ vanishes.  Moreover this is so if only if 
$W$ is homeomorphic to a topologically collapsible polyhedron?
\end{ques}

\section{Preliminaries}
\subsection{Obstructions}\label{obstructions}
According to \cite{Guilb} the manifold $M$ is {\em inward tame at infinity}, if for arbitrarily small neighborhoods of infinity $U$ there exist homotopies $H:U\times [0,1]\to U$ with $H_0$ being identity and $H_1(U)$ having compact closure. 
Alternatively, $M$ is inward tame if and only if arbitrarily small neighborhoods of infinity $U$ are {\em finitely dominated}, namely 
there exist finite complexes $K$ and maps $u:U\to K$ and $d: K\to U$ such that 
$d\circ u$ is homotopic to identity of $U$.

The projective class group functor $\widetilde{K}_0$ associates to a group $\pi$ the 
abelian group $\widetilde{K}_0(\pi)$ of stable isomorphism classes of finitely generated projective left modules over $\Z[\pi]$. 
Wall proved in \cite{Wall1,Wall2} that each finitely dominated CW complex $X$ determines a class 
$\sigma(X)\in \widetilde{K}_0(\pi_1(X))$ which vanishes if and only if 
$X$ has the homotopy type of a finite complex.

We denote by ${\rm Wh}$ the Whitehead functor which associates to a group $\pi$ 
the abelian group $\widetilde{K}_1(\pi)/\pi=K_1(\pi)/(\pm \pi)$, namely the quotient of $GL(\Z[\pi])$ by the subgroup 
generated by the elementary matrices, elements of $\pi$  and $-\mathbf 1$. Note that the subgroup generated by  the elementary matrices coincides with the derived subgroup of $GL(\Z[\pi])$. It is well-known that for every homotopy equivalence of finite CW complexes $f:K\to L$ there exists an element $\tau(f)\in {\rm Wh}(\pi_1(L))$, called the  Whitehead 
torsion of $f$, which vanishes if and only if $f$ is a simple homotopy equivalence.

The previous obstructions have natural extensions to the case of infinite complexes. 
Let $F$ be one of the two functors above. If $X$ is a topological space and $X_i$ denote its connected components we set 
$F(X)=\oplus_{i} F(\pi_1(X_{i}))$. Note that base points are irrelevant as $F$ sends inner automorphisms into identity maps.
 For a complex $X$ one defines the limit of the $F$ functor at the ends $\varepsilon(X)$ of $X$ as the projective limit: 
\[ F(\varepsilon(X))=\lim_{\leftarrow} (F(\pi_1(X-C))_{C\subset X, C \; {\rm compact}}\]
Set also $F^1(\varepsilon(X))$ be the first derived functor of projective limit applied to the inverse system 
$(F(\pi_1(X-C))_{C\subset X, C \; {\rm compact}}$, also called the attenuation of $F$. Recall that the derived limit of an inverse sequence $G_0\stackrel{p_1}{\leftarrow}G_1\stackrel{p_1}{\leftarrow}G_2\stackrel{p_1}{\leftarrow}\cdots$ is the quotient: 
\[ \lim_{\leftarrow}{}^1 (G_i,p_i) =\frac{\prod_{i=0}^{\infty} G_i}{\langle (x_i-p_{i+1}(x_{i+1}))_{i\in \Z_+}, x_i\in G_i\rangle}\]
Note that $F^1(\varepsilon(X))$ vanishes if and only if the inverse system $(F(\pi_1(X-C))_{C\subset X, C \; {\rm compact}}$ 
is equivalent to an inverse system with surjective bonding maps. 

Let now $M$ be an open manifold which we suppose to be inward tame at infinity. 
Choose a compact manifold exhaustion $M_i\subset {\rm int}(M_{i+1})$ of $M$. 
Define $\sigma_{\infty}(M)\in  \widetilde{K}_0(\varepsilon(M))$ to be the class of 
$(\sigma(M-{\rm int}(M_i))_{i\in \Z_+}$. This is a well-defined and independent on the chosen exhaustion (see \cite{CS}).

Let $\tau_i$ denote the image of 
the Whitehead torsion $\tau(M_{i+1}-{\rm int}(M_i), \partial M_i)$ into 
${\rm Wh}(\pi_1(M-{\rm int}(M_i)))$ by the map induced by the inclusion  $M_{i+1}-{\rm int}(M_i)\hookrightarrow  M-{\rm int}(M_i)$.

The Chapman-Siebenmann obstruction $\tau_{\infty}(M)\in {\rm Wh}^1(\varepsilon(M))$ is the image of 
$(\tau_i)_{i\in \Z_+}\in \prod_{i=1}^{\infty}{\rm Wh}(\pi_1(M-{\rm int}(M_i))$ in the quotient 
${\rm Wh}^1(\varepsilon(M))$. 

Note that in \cite{Sie3} there is a more general definition of the obstructions $\sigma_{\infty}$ and $\tau_{\infty}$ 
for proper homotopy equivalences of locally finite complexes, while the one in \cite{CS} mainly concerns $Q$-manifolds.  

\subsection{Semistability}
Recall that the {\em inverse limit} of  an inverse sequence of groups 
\[G_0\stackrel{p_1}{\leftarrow}G_1\stackrel{p_2}{\leftarrow}G_2\stackrel{p_3}{\leftarrow}\cdots\] 
is defined as: 
\[ \lim_{\leftarrow}{} (G_i,p_i) =\{(x_i)_{i\in \Z_+}\in \prod_{i=0}^{\infty} G_i;  p_{i+1}(x_{i+1})=x_i, \; i\in \Z_+\}\]
The inverse limit is an invariant of the pro-equivalence class of the inverse system. 
Here the sequence  $(G_i,p_i)$ and $(H_i,q_i)$ are pro-equivalent, if, after passing to subsequences (namely replacing some $p_i$ by a composition of arrows) and reindexing
there exist homomorphisms $H_{i+1}\to G_{i+1}$ and $G_{i+1}\to H_i$ producing commutative diagrams:
\[\begin{array}{cccccccc}
G_i & & \longleftarrow &    & G_{i+1} &  &  \longleftarrow &\\
& \nwarrow & & \swarrow  &  & \nwarrow & \\
\longleftarrow & &  H_i & &\longleftarrow & & H_{i+1}\\
\end{array} 
 \]

An inverse sequence is {\em stable} if it is pro-equivalent to a constant sequence $(H, id)$. 
An inverse sequence is {\em semistable} if it is pro-equivalent to an inverse sequence $(H_i,q_i)$ where 
all bonding morphisms $q_i$ are surjective. In \cite{Guilb} the author introduced another meaningful related 
notion, as follows. An inverse sequence is {\em perfectly semistable} if it is pro-equivalent to an inverse sequence 
$(H_i,q_i)$ where $H_i$ are finitely presented, the bonding morphisms $q_i$ are surjective and $\ker q_i$ are perfect. 
Note that in the case of perfectly semistable sequence each $\ker q_i$ is the normal closure of a finitely generated subgroup (see \cite{Guilb}, Lemma 2), but this subgroup might not be perfect.

Further we define an inverse sequence to be {\em strongly perfectly semistable} if  it is 
pro-equivalent to an inverse sequence  $(H_i,q_i)$ where $H_i$ are finitely presented, the bonding morphisms $q_i$ are surjective, $\ker q_i$ are perfect and for all $i\leq j$ the subgroup   
$\ker (q_i\circ q_{i+1}\circ \cdots \circ q_{j})\subset H_i$ is the normal closure in $H_i$ of a finitely generated perfect subgroup.

For a one ended open manifold $M$ we consider the  inverse system $\pi_1(\varepsilon(M))$ 
of fundamental groups of $\pi_1(M-K)$, where $K$ are compact subcomplexes of $M$.  
The refined semistability conditions above extend then accordingly to open manifolds by requiring 
 $\pi_1(\varepsilon(M))$ to fulfill them.

The main result of \cite{Guilb-Tin} is the following characterization of pseudo-collarable manifolds: 
\begin{prop}[\cite{Guilb-Tin}]
An open manifold $M$ of dimension $n\geq 6$ is pseudo-collarable if and only if it satisfies the following conditions:
\begin{enumerate}
\item $M$ is inward tame at infinity;
\item $\pi_1(\varepsilon(M))$  is perfectly semistable;
\item Wall's finiteness obstruction $\sigma_{\infty}(M)	\in \widetilde{K}_0(\varepsilon(M))$ vanishes. 
\end{enumerate}
\end{prop}

\subsection{Homology manifolds}
Let $G$ be an abelian group. 
\begin{df}
Let $X$ be a  locally compact topological space $X$ with finite cohomological dimension over $G$. 
Then  $X$ is a  (generalized) {\em homology $G$-manifold} with boundary of dimension $n$ if 
\[ H_i(X, X-\{x\};G)\cong 0, {\rm if }\; i \neq n\]
and for each $x\in X$  the group $H_n(X, X-\{x\};G)$ is either isomorphic to $G$ or $0$. 
The {\em boundary} $\partial X$ of a  homology $G$-manifold with boundary of dimension $n$ consists 
of the set of points $x\in X$ for which $H_n(X, X-\{x\};G)$ vanishes. 
\end{df}

Homology manifolds are the same as homology $\Z$-manifolds.
An example of a homology manifold that is not a topological manifold is the suspension of a 
homology sphere that is not a sphere.

\section{Plus constructions}
A classical construction by Quillen gives a way to kill normal subgroups of the fundamental group of a CW complex while keeping the homology unaltered, as follows:
\begin{prop}
Suppose that $X$ is a finite CW complex and $\pi_1(X)\to \pi$ is a surjective homomorphism onto the finitely presented group $\pi$ perfect kernel. Then there exists a CW complex $Y$ and a continuous map 
$f: X \to Y$, unique up to homotopy equivalence, which induces the given epimorphism $\pi_1(C)\to \pi$ at the level of fundamental groups and is a $\Z[\pi]$-homology equivalence.   
\end{prop}

The space $Y$ is said to be obtained by the plus construction out of $X$ and the given epimorphism and sometimes 
will be denoted by $X^+$. 

The plus construction could also be performed in other categories, e.g. for topological manifolds. 
To this purpose we need the following:
\begin{df}
The compact cobordism $(W, M, N)$ of topological manifolds is a plus cobordism if 
\begin{enumerate}
\item the inclusion $N\hookrightarrow W$ is a simple homotopy equivalence; 
\item the map $\pi_1(M)\to \pi_1(W)$ induced by inclusion is surjective; 
\item the inclusion  $N\hookrightarrow W$ is a $\Z[\pi_1(W)]$-homology equivalence. 
\end{enumerate}
\end{df}

Following Hausmann (\cite{H}, section 3), we have: 
 
 \begin{prop}
 Given a closed topological manifold $M$ of dimension $n\geq 5$ and a surjective homomorphism 
 $\pi_1(M)\to \pi$ onto a finitely presented group with perfect kernel, then there exists an unique 
 plus cobordism $(W,M,N)$ such that the map induced by the inclusion 
 $M\to W$ is $\pi_1(M)\to \pi$, up to homeomorphism relative to $M$. 
 \end{prop}
Moreover $N$ has the homotopy type of the  Quillen plus construction $M^+$. 
If $M$ had a PL or smooth structure, then $W$ is unique up to PL homeomorphism and 
diffeomorphism rel $M$, respectively. 

This construction can be realized also by embedded codimension zero cobordisms in a given manifold with 
boundary (see \cite{Guilb-Tin13}) and all these constructions could be performed by asking 
$N\hookrightarrow W$ be a homotopy equivalence with a prescribed torsion in ${\rm Wh}(\pi_1(W))$ 
(see \cite{Su-Ye}).  

We can relax the plus cobordism definition following Guilbault (see e.g. \cite{Guilb}), 
to a notion essential for pseudo-collarable manifolds: 

\begin{df}
The compact cobordism $(W, M, N)$ of topological manifolds is a one sided $h$-cobordism if 
 the inclusion $N\hookrightarrow W$ is a  homotopy equivalence. 
\end{df}
It is a well-known consequence of a result  of Hausmann   (see \cite{H2}, Lemma 2.0, \cite{DT1}, Lemma 2.5., \cite{Guilb}, Lemma 6) that the map $\pi_1(M)\to \pi_1(W)$ induced by inclusion is surjective with perfect kernel and 
 $N\hookrightarrow W$ is a $\Z$-homology equivalence, i.e. $H_*(W,M)=0$.

\section{Necessary conditions for CAT(0) metrics}
The aim of this section is to prove the first part of Theorem \ref{thm:CAT1}, which we restate here 
for completeness:
\begin{thm}
An open  topological $n$-manifold $M$ with $n\geq 5$  carrying a complete polyhedral  CAT(0) metric satisfies the following conditions:
\begin{enumerate}
\item $M$ is pseudo-collarable;
\item The Chapman-Siebenmann obstruction $\tau_{\infty}(M)\in {\rm Wh}^1(\varepsilon(M))$ vanishes; 
\item  $\pi_1(\varepsilon(M))$ is strongly perfectly semistable.
\end{enumerate}
\end{thm}

\subsection{CAT(0) metrics necessitate pseudo-collarability}\label{necessary} We  first argue  that the condition of pseudo-collarability is necessary. 
\begin{prop}\label{pseudocollar}
An open topological $n$-manifold $M$, with $n\geq 5$,  which admits a polyhedral CAT(0) metric is pseudo-collarable. 
\end{prop}
This result is known for $n\geq 6$: it follows from Remark 5 of \cite{Guilb} and the main Theorem in \cite{Guilb-Tin}.  Here is an alternative proof, covering the case $n=5$, as well. 
\begin{proof}
The closed metric ball centered at $p\in M$ of radius $r$ is denoted by $B(p,r)$. 
Its interior ${\rm int}(B(p,r))$ is the open metric ball of radius $r$. 

Now,  Alexandrov proved  (\cite{BN}, Prop. 8.2-8.3  and \cite{Th}, Prop. 2.) that any geodesic space endowed with a 
CAT(0) metric which is a homology manifold also has the geodesic extension property, namely every geodesic segment can be extended to a  bi-infinite geodesic (see  also \cite{BH}, ch.\ II.5, Prop.\ 5.12).  This property is also called geodesic completeness. Therefore 
the metric sphere of radius $r$ coincides with the frontier $\partial B(p,r)$, namely 
with the set of points $q\in B(p,r)$ such that  ${\rm int}(B(q,\varepsilon))\cap (M\setminus B(p,r))\neq\emptyset$, 
for every $\varepsilon >0$.

\begin{lem}
The metric sphere  $\partial B(p,r)$ is homotopy equivalent to $\overline{M\setminus B(p,r)}$.
\end{lem}
\begin{proof}
Indeed $\partial B(p,r)$ is a strong deformation retract of $\overline{M\setminus B(p,r)}$ (see \cite{BH}, ch.\ II.2, Prop.\ 2.4.(4)). 
\end{proof}

We shall see next that there exists a manifold approximation of $\partial B(p,r)$ sharing the same property. 

\begin{lem}
When the CAT(0) metric is polyhedral, the metric sphere $\partial B(p,r)$ is a polyhedral homology $(n-1)$-manifold having the homology of a  $(n-1)$-sphere. 
\end{lem}
\begin{proof}

According to (\cite{Th}, Prop. 2.7) $B(p,r)$ is a homology $n$-manifold with boundary, whose boundary $\partial B(p,r)$ as homology manifold coincides with the metric sphere sphere of radius $r$ and hence with its frontier. According to Mitchell theorem from \cite{Mit} the metric sphere $\partial B(p,r)$ is a homology $(n-1)$-manifold. Since $B(p,r)$ has the homology of a $n$-ball, its boundary $\partial B(p,r)$ has the homology of a $(n-1)$-sphere. 

Further the metric sphere is polyhedral since its intersection with every cell of the triangulation $\Delta$ is a polyhedron.  
Indeed $B(p,r)$ is convex, namely the geodesic segment determined by two points of it is 
contained in $B(p,r)$. Therefore the intersection $B(p,r)\cap e$ with a cell $e$ of $\Delta$ is a 
convex subset of $e$ and hence it bounds a polyhedron. Now $B(p,r)$ has non-empty intersection with only finitely many cells of $\Delta$, because $\Delta$ is locally finite. Thus $\partial B(p,r)$ is a finite  union of  
polyhedra which pairwise intersect along subpolyhedra and hence a polyhedron itself. 
\end{proof}

Next we will extend the result of Ferry (see \cite{Ferry2,Ferry}) to an approximation theorem 
of resolvable generalized homology manifolds in codimension one. We have first: 
\begin{lem}\label{resolvable}
For generic radii $r$ the polyhedron $\partial B(p,r)$ admits a resolution. 
\end{lem}
\begin{proof}
Quinn's resolution theorem (\cite{Weinberger}) states that there exists a locally defined obstruction invariant  in $1+8\Z$ 
which detects precisely when a generalized homology sphere has a resolution. As above, the intersection 
$B(p,r) cap e$ with a cell $e$ of $\Delta$ is convex . Furthermore,  since the metric is piecewise smooth, then  
for generic $r$ the frontier of  $B(p,r)\cap e$ is a convex hypersurface.  The later contains an open dense set 
which is a piecewise linear submanifold. Therefore $\partial B(p,r)$ contains manifold points.  
In particular Quinn's obstruction is trivial. 
\end{proof}

We now want to prove that there exists arbitrarily close approximations of 
the generalized homology manifold $\partial B(p,r)$ by locally flat topological submanifolds of $M$. 
Namely, there exists by Lemma \ref{resolvable} a closed $(n-1)$-manifold $S$ endowed with 
a surjective cell-like map  $g:S\to \partial B(p,r)$. 

\begin{lem}\label{approximation} 
For every $\varepsilon >0$ there exists a topologically flat 
embedding $h_{\varepsilon}: S\to M$, such that $h_{\varepsilon}(S)$ is  $\varepsilon$-close to $\partial B(p,r)$. 
Moreover, there exists an ambient homotopy $H:M\times [0,1]\to M$, with the property that 
$H_1={\rm id}$,  $H_t$ is a homeomorphism for any $t>0$, 
$H_{\varepsilon}(h_1(S))=h_{\varepsilon}(S)$ and $H_0(h_1(S))=\partial B(p,r)$. 
\end{lem}
\begin{proof}
{
Observe that $\partial B(p,r)$ is separating $M$. 
Consider the ENR $M'=B(p,r)\cup \partial B(p,r)\times [0,1] \cup M-{\rm int}(B(p,r))$, which is a generalized homology 
manifold. There is a  proper cell-like map $q: M'\to M$  which collapses $\partial B(p,r)\times [0,1]$ to 
$\partial B(p,r)$. Further $M'$ admits a resolution, as it contains manifolds points; namely there exists a proper cell-like map  
$p: P\to M'$ from a manifold $P$. Let $f:P \to M$ be the composition $q\circ p$. 
Then $f$ is a proper cell-like map. By a classical result of Siebenmann (see \cite{Sie2}, Approximation Thm. A) 
$f$ is a limit homeomorphism. This means that there exists a level preserving cell-like map 
$F: P\times [0,1]\to M\times [0,1]$ such that $F(x,t)=(f_t(x),t)$, where $f_t:P\to M$ for $0\leq t <1$ are 
homeomorphisms $\varepsilon$-close to $f$ and $f_1=f$.   
}

{   
As $\partial B(p,r)$ is a codimension one compact polyhedron, the combinatorics of the intersections 
with a triangulation subjacent to the polyhedral complex $\Delta$, and hence its homeomorphism type, 
will not change in a small neighborhood of the generic $r$. It follows that 
$\partial B(p,r)\times (-\varepsilon,\varepsilon)$ is embedded in $M$ and hence it is a manifold.  
The product map $g \times {\rm id}:S\times (\varepsilon,\varepsilon)\to
\partial B(p,r)\times (\varepsilon,\varepsilon)$ is proper and cell-like.
Since both polyhedra are topological manifolds, $g\times {\rm id}$ is a 
limit homeomorphism, by the result of Siebenmann cited above. 
By the same argument map $p|_{p^{-1}(\partial B(p,r)\times (\varepsilon,\varepsilon))}$ is also a limit homeomorphism. 
Therefore there exists a codimension zero embedding $g_{\varepsilon}:S\times (\varepsilon,\varepsilon)\to P$. 
}

{
It follows that $f_{1-\varepsilon}\circ g_{\varepsilon}(S\times \{0\})$ is a locally flat approximation 
of $\partial B(p,r)$ (see also (\cite{Ferry}, Thm.1).   We put $h_{\varepsilon}=f_{1-\varepsilon}\circ g_{\varepsilon}$. 
The ambient homotopy $H$ is constructed from $F$, by identifying $P$ and $M$ by means of $F_0$.  
}
\end{proof}
Let $V$ be the closure of the unbounded component of $M\setminus h_1(S)$. 
It follows that for all $j\geq 1$ we have: 
\[\pi_j(V,h_1(S))\cong\pi_j(H_t(V), H_t(h_1(S))), \; {\rm for \; all }\; t>0.\]  
As $h_1(S)$ has codimension one  and $H_1$ is a hereditary homotopy equivalence we can pass to the limit $t\to 0$ to obtain: 
 \[\pi_j(V,h_1(P))\cong \pi_j(\overline{M\setminus B(p,r)}, \partial B(p,r))=0.\]
This shows that $V$ is a manifold pseudo-collar, as claimed. 
\end{proof}


\subsection{CAT(0) metrics need trivial Chapman-Siebenmann obstruction}
Any locally compact ANR, in particular a manifold $M$ with a CAT(0) metric has a $\mathcal Z$-compactification (see \cite{AG99}, Ex.6, Rk.1). It follows that the product $M\times Q$ with the Hilbert cube $Q$ has a $\mathcal Z$-compactification. 
The main theorem of \cite{CS} shows that $\tau_{\infty}(M\times Q)=0$. On the other hand $\tau_{\infty}$ is a proper homotopy invariant and hence $\tau_{\infty}(M )=0$.

 \subsection{Strong perfect semistability}
 The {\em geodesic contraction} is the map $c_r:M-{\rm int}(B(p,r))\to \partial B(p,r)$  sending a point $q\in M-{\rm int}(B(p,r))$ into the point $c_r(q)\in \partial B(p,r)$ lying on the geodesic segment joining $p$ and $q$ which is at distance $r$ from $p$.  Its restriction to a metric sphere provides 
 the geodesic contraction map $c_{R,r}:\partial B(p,R) \to \partial B(p,r)$, for any $r<R$.  
 
 Moreover, the map $c_{R,r}$ has acyclic point inverses (see \cite{DJ}, Proof of Thm. 3d.1 and 
 \cite{Th}, Cor. 2.10) and hence it is a degree one between  homology manifolds with the 
 homology of a sphere, for all $r<R$. 
  
The {\em shadow} of a point $q\in \partial B(p,r)$ is the set $c_r^{-1}(q)\subset  M-{\rm int}(B(p,r))$. 

In the case of polyhedral CAT(0) metrics the shadow and hence the fibers of the geodesic 
contractions $c_{R,r}$ are determined by local data called {\em infinitesimal shadow}, as defined in (\cite{DJ}, section 2d). 
Let $\gamma:(-\varepsilon, \varepsilon)\to M$ be a germ of geodesic  at $\gamma(0)=p$. 
We denote by $Lk(p,\Delta)$ the {\em link} of $p$ in the simplicial complex $\Delta$.  Then $\gamma$ defines two 
points $\gamma_p^{\pm}\in Lk(p,\Delta)$, which are the incoming and outgoing tangent vectors along $\gamma$ at $p$.
Then, for $p\in \Delta$ and $v\in Lk(p,\Delta)$  the infinitesimal shadow $sh(p, v)\subset Lk(p,\Delta)$ is defined as 
\[ sh(p, v)= \{w; {\rm there \; exists \; geodesic \; germ }\; \gamma\;  {\rm at \; } p \; {\rm such \; that} \; 
\gamma_p^-=v, \gamma_p^+=w\}\subset Lk(p,\Delta)\]

\begin{lem}
If the CAT(0) metric is polyhedral, then for generic $r<R$ the fibers of the geodesic contraction 
$c_{R,r}: \partial B(p,R) \to \partial B(p,r)$ are  acyclic ANR. 
\end{lem}
\begin{proof}
According to (\cite{DJ}, Lemma 2d.1) the infinitesimal shadow  $sh(p,v)$ is the 
complement of the open metric ball of radius $\pi$ in the link 
$Lk(p,\Delta)$, endowed with its  piecewise spherical CAT(1) metric. 
Then both the metric disk and $sh(p,v)$ are subpolyhedra of $Lk(p,\Delta)$. 
Moreover, the open metric disk of radius $\pi$ is a contractible homology manifold with boundary. 

Consider now  $q\in \partial B(p,r)$  and  $z\in c_r^{-1}(q)$. Let $\gamma$ be the geodesic segment joining 
$q$ to $z$, which intersects only finitely many cells of $\Delta$, say in order $e_1,e_2,\ldots,e_k$. 

By induction on $j$ we show that $c_r^{-1}(q)\cap e_j$ is a polyhedron. 
Now, either $e_1\prec e_2$ or $e_2\prec e_1$. In the first case 
$c_r^{-1}(q)\cap e_1=\{p\}$. In the second case $c_r^{-1}(q)\cap e_1$ is PL homeomorphic to 
the intersection of the cone $C(sh(p, \gamma_p^-)\cap e_2)$ and hence a polyhedron. 
Further, for the induction step, assume the claim holds for  the first $j$ cells.

Now, for each point $x$ of  the polyhedron $c_r^{-1}(q)\cap e_j$ 
there is a geodesic segment $\gamma(x)$ joining $x$ and $p$, which 
determines an incoming tangent vector $\gamma(x)_p^-=v(x)\in Lk(x, \Delta)$. 
For all $x$ which belong to the open cell $e_j$ the links $Lk(x,\Delta)$ are pairwise isometric and hence 
$sh(x,\gamma(x)_x^-)$ are pairwise isometric spherical polyhedra.   
Therefore  
$c_r^{-1}(q)\cap e_{j+1}$ is PL homeomorphic to either the join of two polyhedra
\[(c_r^{-1}(q)\cap e_j)* (sh(x,\gamma(x)_x^-)\cap e_{j+1}),  {\rm  \; if} \; e_j\prec e_{j+1}\] 
or the intersection 
\[ \left( (c_r^{-1}(q)\cap e_j)* (sh(x,\gamma(x)_x^-)\cap e_{j})\right) \cap e_{j+1}, \; {\rm when }\; e_{j+1}\prec e_j\]

Since $B(p,r)$ is convex and compact,  $c_r^{-1}(q)\cap B(p,r)$ is a finite union of polyhedra and hence an ANR, actually a compact metrizable ENR. Moreover, the intersection $c_r^{-1}(q)\cap \partial B(p,r)$  is also a finite union of polyhedra 
and hence a compact metrizable ENR. 

This implies that point inverses of the geodesic contraction $c_{R,r}$ are  acyclic ANR. 
\end{proof}

We shall now use the following extension of a result from (\cite{DT1}, Prop.4.8):

\begin{lem}
Suppose that $f:N_1\to N_2$ is an acyclic map between  $(n-1)$-dimensional homology manifolds without boundary such that $f^{-1}(y)$ is an ANR for every $y\in N_2$. Then the kernel  $\ker f_*$ of the map 
$f_*:\pi_1(N_1)\to \pi_1(N_2)$ induced 
by $f$ at the level of fundamental groups is the normal closure of a finitely generated perfect group. 
\end{lem}
\begin{proof}
The proof is similar to that presented in \cite{DT1} for topological manifolds but we include it here for the sake of completeness. For every $y\in N_2$ the preimage $f^{-1}(y)$ being an ANR by our assumptions admits an open 
neighborhood $U\subset N_1$ which deformation  retracts onto $f^{-1}(y)$. As $f$ is surjective, since acyclic, 
the collection of open sets $U$ is a covering of $N_1$. As $N_1$ is compact we can extract a finite cover 
$\{U_i\}$ corresponding to the points $y_i\in N_1$.
We can join a base point $z_0\in N_1$ with $f^{-1}(y_i)$ by means of pairwise disjoint arcs which only intersect 
$\cup_{i}f^{-1}(y_i)$ at their end points. Denote by $Y$ the union of $\cup_{i}f^{-1}(y_i)$ with 
these arcs based at $z_0$. 

Then $Y$ is a compact acyclic ANR. 
By a theorem of West (\cite{West}) any compact metrizable ANR  is homotopy equivalent to a finite cell-complex and 
hence each $\pi_1(f^{-1}(y_i)$ is finitely presented and hence $\pi(Y)$ is finitely presented and perfect. 

Moreover, $\ker f_*$ is normally generated by the image of $\pi_1(Y)$ under the map 
$\pi_1(Y)\to \pi_1(N_1)$ induced by the inclusion $Y\hookrightarrow N_1$. 
Thus $\ker f_*$ is the normal closure of a finitely generated perfect group. 
\end{proof}

We obtained so far a compact exhaustion of $M$ by homology manifolds for which the fundamental pro-group at infinity 
is given by a sequence of surjective homomorphisms, each bonding map having its kernel normally generated by a finitely generated perfect subgroup. 

The homology manifolds arising as boundaries admit arbitrarily close approximations by locally flat 
topological submanifolds of $M$, by Lemma \ref{approximation}. Using notation from this lemma, we know that 
for $r_1 >r_2$ there exist topologically flat embeddings $h_{\varepsilon}:S_i \to M$ such that 
$h_{\varepsilon}(S_i)$ are $\varepsilon$-close to $\partial B(p,r_i)$ and 
an ambient homotopy $H:M\times [0,1]\to M$ such that  $H_1=id$, 
$H_{\varepsilon}$  is a homeomorphism for every $\varepsilon >0$ and 
$H_{\varepsilon}(h_1(S_i))=h_{\varepsilon}(S_i)$, while $H_0((h_1(S_i))=\partial B(p,r_i)$. 
Let $Z$ be the manifold bounded by $h_1(S_1)\sqcup h_1(S_2)$ and $V$ the closure of the unbounded 
component of $M-h_1(S_2)$. Recall from the last lines of the proof of Proposition \ref{pseudocollar} 
that $V$ is a pseudo-collar.

As $H_1$ is a hereditary homotopy equivalence we have an identification between the sequence of maps: 
\[\pi_1(\partial B(p,r_1))\to \pi_1(B(p,r_1) -{\rm int}(B(p,r_2))) \to 
\pi_1(M-{\rm int}(B(p,r_2)))\to \pi_1(\partial B(p,r_2))\]
and 
\[\pi_1(h_1(S_1))\to \pi_1(Z) \to \pi_1(V)\to \pi_1(h_1(S_2))\]
Therefore the map $\pi_1(h_1(S_1))\to \pi_1(h_1(S_2))$ is surjective and its kernel is 
normally generated by a finitely generated perfect group.  
Thus we can replace replace metric spheres by their topologically flat approximation, while keeping 
the same fundamental pro-group at infinity for the associated exhaustions. 

This implies that $M$ is strongly perfectly semistable.

\section{Exhaustions  of pseudo-collars by plus cobordisms}\label{pluscobordisms}
\subsection{Exhaustions of pseudo-collars with trivial $\tau_{\infty}$}

We start by recalling the following result of Guilbault: 
\begin{prop}[\cite{Guilb}]
An open manifold is pseudo-collarable if and only if it is union of one-sided h-cobordisms with disjoint 
interiors. 
\end{prop}
This section aims at refining this characterization as follows: 
\begin{prop}
An open manifold is pseudo-collarable and its Chapman-Siebenmann obstruction 
$\tau_{\infty}(M)\in {\rm Wh}^1(\varepsilon(M))$ vanishes if and only if 
it is the union of  one-sided $h$-cobordisms with trivial torsion and disjoint 
interiors. 
\end{prop}

The if part is trivial, as unions of  one-sided cobordisms with vanishing torsion have trivial 
Chapman-Siebenmann obstruction.
 
For the converse  we first record, following \cite{Guilb,Guilb-Tin}:
 \begin{lem}\label{psdc}
If the open contractible manifold  $M$ is pseudo-collarable  then there exists  
an exhaustive filtration $M_{i}$, $i\geq 0$ of $M$ with the following properties:
\begin{compactenum} 
\item $M_{i}$ are  compact contractible manifolds;
\item the inclusion maps  $\overline{M\setminus M_j}\hookrightarrow \overline{M\setminus M_{i}}$ for $j>i$ 
induce surjections at the level of fundamental groups, and
\item the inclusions $\partial M_{i}\hookrightarrow \overline{M\setminus M_{i}}$ induce isomorphisms at the level of fundamental groups.
\end{compactenum}
\end{lem}

Recall now from \cite{Guilb} that any pseudo-collar $W$ can be written as the union of 
1-sided h-cobordisms $W_i$ with disjoint interiors. This means that $W_i$ is a cobordism with left boundary 
$J_i$ and right boundary $J_{i+1}$, so that $J_1=\partial W$, with the property that 
$J_i\subset W_i$ is a homotopy equivalence. The 1-sided h-cobordism $W_i$ is said to be a {\em plus cobordism} 
(see \cite{Q,Rolland}) if the inclusion $J_i\subset W_i$ is a simple homotopy equivalence, namely the torsion 
$\tau(W_i, J_i)$ vanishes in the Whitehead group $Wh(\pi_1(J_i))$. 
One key property needed in the construction below is the following: 

\begin{lem}\label{plusdecomp}
A pseudo-collar manifold $W$ is the union of plus cobordisms with disjoint interiors if (and only if) $\tau_{\infty}(W)=0$.
\end{lem}
\begin{proof}
Our proof follows closely the one given in \cite{CS} for $Q$-manifolds. We start with: 
\begin{lem}
 Let $N$ be a closed $(n-1)$-manifold, $n\neq 4$ and $\mu\in {\rm Wh}(\pi_1(N))$.
 Then there is  a decomposition of $N\times [0,1]=Z_1\cup Z_2$ into two $h$-cobordisms 
 $Z_1$ and $Z_2$ with disjoint interiors such that 
 \[ \tau(Z_1,N\times\{0\})=\mu, \;  \tau(Z_2,Z_1\cap Z_2)=-\mu\]
  \end{lem}
\begin{proof}
There exist $h$-cobordisms $Z_1$ and $Z_2$ with prescribed torsions. Their composition then has trivial torsion: 
\[\tau(Z_1\cup Z_2, N\times\{0\})=\tau(Z_1,N\times\{0\})+\tau(Z_2,Z_1\cap Z_2)= 0\]
By the topological s-cobordism theorem $Z_1\cup Z_2$ is homeomorphic to $N\times [0,1]$. 
 \end{proof}

Assume now that we have a filtration $M_i$ of the pseudo-collar $M$ with the property that 
$M_{i+1}-{\rm int}(M_i)$ are one sided $h$-cobordisms (see e.g. \cite{Guilb}, Prop.2). Let $\tau_i\in {\rm Wh}(M-{\rm int}(M_i))$ denote the image 
of $\tau(M_{i+1}-{\rm int}(M_i), \partial M_i)\in {\rm Wh}(\pi_1(M_{i+1}-{\rm int}(M_i))$
in the group ${\rm Wh}(M-{\rm int}(M_i))$, by means of the inclusion induced homomorphism.
 
By hypothesis $\tau_{\infty}(M)=0$ and hence there exist 
$(\mu_1,\mu_2,\ldots)\in \prod_{i=1}^{\infty} {\rm Wh}(M-{\rm int}(M_i))$, such that for every $i$ 
\[ \mu_i- p_i(\mu_{i+1})= \tau_i\]
where $p_i: {\rm Wh}(M-{\rm int}(M_{i+1}))\to {\rm Wh}(M-{\rm int}(M_i))$ is the induced homomorphism. 

The only reason to consider this obstruction is the fact that although the group maps are surjective the corresponding 
maps between the Whitehead groups is not necessarily surjective.

Recall that $\partial M_i \hookrightarrow M-{\rm int}(M_i)$ is a homotopy equivalence.
Let then $\mu_i'\in {\rm Wh}(\partial M_i)$ be a class whose image by the inclusion induced homomorphism 
is $\mu_i\in  {\rm Wh}(M-{\rm int}(M_i))$. The previous lemma gives us a decomposition 
of a collar $\partial M_i\times [0,1]\subset M_{i+1}-{\rm int}(M_i)$ as the union of two $h$-cobordisms 
$Z_i^1\cup Z_i^2$ with disjoint interiors, such that 
\[ \tau(Z_i^1, \partial M_i)=\tau_i', \; \tau (Z_i^2, Z_i^1\cap Z_i^2)=-\tau_i'\]
We set now $M_{i}'=M_{i}\cup Z_i^1$. Then $\partial M_i'\hookrightarrow M_{i+1}'-{\rm int}(M_i')$ is 
a homotopy equivalence. 

By the formula of the torsion of a composition we have: 
\begin{eqnarray*} \tau(M_{i+1}'- {\rm int}(M_i'), \partial M_i')& =&\tau(Z_i^2, \partial M_i)+ \tau(M_{i+1}-{\rm int}(M_i), \partial M_i) + j_*\tau(Z_{i+1}^1, \partial M_{i+1})\\
&=&-\mu_i+ \tau_i+p_{i+1}(\mu_{i+1})=0\end{eqnarray*}
where $j_*$ is the map induced from inclusion $Z_{i+1}^1\hookrightarrow M-{\rm int}(M_i)$. 
It follows that $M$ is the union of one-sided $h$-cobordisms with trivial torsion. 
\end{proof}

\subsection{Exhaustions by plus cobordisms} 
The aim of this section is to provide the following key ingredient in the proof of Theorem \ref{thm:CAT1}:

\begin{prop}\label{pluscobordism}
An open contractible $n$-manifold $W$, $n\geq 6$, which  is pseudo-collarable,  has strongly perfectly semistable fundamental group at infinity and has a vanishing Chapman-Siebenmann obstruction $\tau_{\infty}$ 
is the union of plus cobordisms with disjoint interiors. Moreover, the plus cobordisms are homeomorphic 
to mapping cylinders of acyclic maps between the boundaries. 
\end{prop}

\begin{rem}
Chapman and Siebenmann considered in \cite{CS} the notion of {\em infinite mapping cylinder} of an inverse sequence of 
compact metric spaces $f_i:X_{i+1}\to X_i$, by sewing together the mapping cylinders $C(f_i)$ along their naturally identified boundaries. Proposition \ref{pluscobordism} states that $W^n$ is an infinite mapping cylinder  of an inverse sequence of acyclic maps  with $X_0$ being a point and $X_i$  closed $(n-1)$-manifolds. 
\end{rem}

\begin{lem}\label{sperfectplus}
Assume that the open manifold $M$ is strongly perfectly semistable at infinity and union of plus cobordisms with disjoint interiors.
Then there exists a compact exhaustion  $M_i$ such that 
$M_{i+1}-{\rm int}(M_i)$ are one sided $h$-cobordisms  with trivial torsion and moreover  each 
$\ker(\pi_1(\partial M_{i+1})\to \pi_1(\partial M_i))$ is the 
normal closure in $\pi_1(\partial M_{i+1})$ of some finitely generated perfect subgroup.
\end{lem}
\begin{proof}
We start with a compact exhaustion $M_i$ such that $\partial M_i\hookrightarrow M_{i+1}-{\rm int}(M_i)$ are homotopy equivalences. 
We can change the exhaustion by passing to subsequences and relabelling such that 
the  kernel $\ker(\pi_1(\partial M_{i+1})\to \pi_1(\partial M_i))$ is the 
normal closure in $\pi_1(\partial M_{i+1})$ of some finitely generated perfect subgroup, while keeping the property that 
$\partial M_i\to M_{i+1}-{\rm int}(M_i)$ are homotopy equivalences. We are then in the situation of lemma 
\ref{plusdecomp}. We alter the decomposition into one sided $h$-cobordisms by adjoining 
$h$-cobordisms. However these changes preserve the 
fundamental groups involved and hence the new decomposition fulfills all required conditions. 
\end{proof}

Thus $\partial M_i\hookrightarrow M_{i+1}-{\rm int}(M_i)$ are simple homotopy equivalences. 

\subsection{Mapping cylinders} 
Further, we recall the following key result  from (\cite{DT1}, Thm.5.2), which we record below:

\begin{prop}\label{Daverman-Tinsley}
Suppose that $(M,N_1,N_2)$ is an $(n)$-dimensional cobordism, $n\geq 6$, such that 
$N_2\hookrightarrow M$ is a homotopy equivalence and the kernel of the map induced by 
inclusion $\pi_1(N_1)\to \pi_1(M)$ is the normal closure in $\pi_1(M)$ of a finitely generated 
perfect group. Then there exists an acyclic map $f:N_1\to N_2'$ to a closed manifold 
$N_2'$ such that $M$ is homeomorphic to $M(f)\cup _{N_2'\times \{1\}}M'$, where 
$M(f)$ is the topological mapping cylinder of $f$ and $(M',N_2'\times\{1\}, N_2)$ is an h-cobordism. 
\end{prop}

This shows that  every cobordism $M_{i+1}-{\rm int}(M_i)$ obtained from an exhaustion as provided by Lemma   \ref{sperfectplus} is homeomorphic to the composition of the mapping cylinder  $C_i$ of some acyclic map 
$\partial M_{i+1}\to N_i$ composed with an $h$-cobordism $Z_i$ with boundary $\partial Z_i=N_i\sqcup \partial M_i$.  
Here $N_i$ is some closed manifold homotopy equivalent to $\partial M_i$. 
 
Now, after identifying the groups $\pi_1(N_i)$, $\pi_1(\partial M_i)$, $\pi_1(M_{i+1}-{\rm int}(M_i))$ and $\pi_1(Z_i)$ 
we have in ${\rm Wh}(\pi_1(\partial M_i))$:
\[ \tau (M_{i+1}-{\rm int}(M_i), \partial M_i)=\tau(C_i, N_i) +\tau(Z_i,\partial M_i)\]
On the other hand 
\[ \tau (M_{i+1}-{\rm int}(M_i), \partial M_i)=0\]
because  $M_{i+1}-{\rm int}(M_i)$  is a plus cobordism. 

We will use now the following: 
\begin{lem}\label{west}
If $f:M\to N$ is an acyclic map whose topological mapping cylinder $M(f)$ is a manifold, then 
the retraction map $\pi: M(f)\to N$ is a simple homotopy equivalence and hence 
$\tau(M(f), N)=0$.  
\end{lem}
\begin{proof}
We use the Chapman's simple homotopy type of compact ANR spaces (see \cite{Chapman}). 
Topological manifolds have the homotopy type of a CW complex, by Kirby and Siebenmann and so 
the product of a compact topological manifold and the Hilbert cube $Q$ is a compact Hilbert cube manifold. 
In our case the manifolds we consider are also triangulable and we don't need the full force 
of the Kirby-Siebenmann theorem. 
According to a deep result of (\cite{West1}, Thm. 2 and section 4, \cite{Sie4}, Thm. 3.4), the  retraction map 
$\pi\times id_Q: M(f)\times Q\to N\times Q$ is homotopic to a 
homeomorphism of Hilbert cube manifolds. Further, Chapman's theorem (\cite{Chapman2,Chapman3}) says that 
$\pi$ is a simple homotopy equivalence. As the inclusion $N\hookrightarrow M(f)$ is a 
homotopy inverse for the retraction map, we derive that $\tau(M(f),N)=0$. 

Alternatively, we can use the fact that a cell-like map between compact ANR's is a simple homotopy equivalence 
(see \cite{Lacher}, Thm. 4.3, following Chapman and West). The retraction map 
$M(f)\to N$ is obvious cell-like, as the point preimage of $y\in N$ is the cone over $f^{-1}(y)$. 
\end{proof}

Now $\tau(C_i,N_i)=0$ and hence $\tau(Z_i,\partial M_i)=0$. By the s-cobordism theorem 
$Z_i$ is homeomorphic to a product and hence  $M_{i+1}-{\rm int}(M_i)$  is homeomorphic to a mapping cylinder 
$C_i$. 

Note that we don't know if we can perturb the acyclic map $f$ to an acyclic PL map;   if this were true, then the its simplicial mapping cylinder would PL collapse onto $\partial M_i$.

\subsection{Acyclic maps}
We will use later the description of $f$ following \cite{DT1} and \cite{Daverman}. 
Let $Q_{i+1}\subset \pi_1(\partial M_{i+1})$ be a finitely generated perfect subgroup  whose normal closure within 
$\pi_1(\partial M_{i+1})$ is $\ker(\pi_1(\partial M_{i+1})\to \pi_1(\partial M_{i}))$.
By Hausmann's trick  there exists $Q_{i+1}^*$ a  finitely presented perfect group of deficiency $0$ equipped with a surjection onto $Q_{i+1}$ (see \cite{H}, section 2.1) . Let $D_{i+1}$ be a  $2$-dimensional complex associated to a balanced presentation of $Q_{i+1}^*$, which is then acyclic. We embed $D_{i+1}$ within $\partial M_{i+1}$ such that the induced map on fundamental groups is the composition $Q_{i+1}^*\to Q_{i+1}\subset \pi_1(\partial M_{i+1})$. 
Then the boundary $\partial \mr{N}_{\partial M_{i+1}}(D_{i+1})$ of a regular neighborhood 
$\mr{N}_{\partial M_{i+1}}(D_{i+1})$ of $D_{i+1}$ within $\partial M_{i+1}$ is a 
codimension one homology sphere.

Now, the inclusion induced map is an isomorphism 
\[ \pi_1(\partial \mr{N}_{\partial M_{i+1}}(D_{i+1}))\to \pi_1(\mr{N}_{\partial M_{i+1}}(D_{i+1}))\]
Indeed any loop in $\mr{N}_{\partial M_{i+1}}(D_{i+1}))$ based at a point of the boundary can be 
homotoped out of $D_{i+1}$, by general position. Since $\mr{N}_{\partial M_{i+1}}(D_{i+1}))-D_{i+1}$
is homeomorphic to $\partial \mr{N}_{\partial M_{i+1}}(D_{i+1}))\times [0,1)$ we can further homotope 
the loop onto $\partial \mr{N}_{\partial M_{i+1}}(D_{i+1}))$, proving that the map above is surjective. 
Further, a null-homotopy 2-disk mapped properly into $(\mr{N}_{\partial M_{i+1}}(D_{i+1})),\partial \mr{N}_{\partial M_{i+1}}(D_{i+1}))$ can be homotoped off $D_{i+1}$ by general position and hence into 
$\partial \mr{N}_{\partial M_{i+1}}(D_{i+1}))$. This yields the injectivity of the homomorphism above. 

We can further homotope the map $D_{i+1}\to \mr{N}_{\partial M_{i+1}}(D_{i+1}))$ to 
an embedding  of $D_{i+1} $ in the boundary, namely such that 
its image is $D_{i+1}^*\subset \partial \mr{N}_{\partial M_{i+1}}(D_{i+1}))$,
and the induced map on fundamental groups is an isomorphism. 

Eventually we consider a collar $\partial \mr{N}_{\partial M_{i+1}}(D_{i+1}))\times [0,1]\subset \partial M_{i+1}$ 
of the boundary  and a Cantor set $C\subset [0,1]$. 
The sets $D_{i+1}^*\times C$ form the set of nondegenerate elements of a upper semi-continuous decomposition $\mathcal U$ of $\partial M_{i+1}$. The associated quotient space $\partial M_{i+1}/\mathcal U$ is the quotient by the equivalence 
relation induced by $\mathcal U$, namely two points are identified if and only if they belong to the same set of the decomposition. We have then: 

\begin{prop}[\cite{DT1} Thm.4.3, \cite{Daverman}, section 2, \cite{DT4}]\label{quot}
The space $\partial M_{i+1}/\mathcal U$ is a topological manifold and the mapping cylinder 
of the quotient map $f_i:\partial M_{i+1}\to \partial M_{i+1}/\mathcal U$ 
is a topological manifold. 
\end{prop}

\begin{proof}[Proof of Proposition \ref{pluscobordism}]
By Proposition \ref{Daverman-Tinsley} and the discussion above it follows that 
$M_{i+1}-{ \rm int}(M_i)$ is homeomorphic to the topological 
mapping cylinder $M(f_i)$ of an acyclic map 
$f_i: \partial M_{i+1}\to \partial M_i$, which is a quotient map as  described in Proposition \ref{quot}. 

Since it is acyclic $f_i$ is a $\Z[\pi_1(\partial M_i)]$ homology equivalence. 
On the other hand  $f_i$ factors as $\partial M_{i+1}\hookrightarrow M_{i+1}-{\rm int}(M_i)\to \partial M_i$, 
where the second map is the strong deformation retract of the mapping cylinder $M(f_i)$ onto its target boundary $\partial M_i$. This implies that the inclusion  $\partial M_{i+1}\hookrightarrow M_{i+1}-{\rm int}(M_i)$ is a  $\Z[\pi_1(\partial M_i)]$ homology equivalence, and hence this cobordism is a plus cobordism. 
\end{proof}

\subsection{Collapsibility of unions of plus cobordisms with strongly perfectly semistable group at infinity}

We found above that  $M$  is endowed with an exhaustion 
by plus-cobordisms, i.e. an ascending filtration by compact 
contractible submanifolds $M_{i}\subset M$  such that $M_{i+1}-{\rm int}(M_{i})$ is a plus cobordism for every $i\ge 0$.

The main result of  the previous section states that 
each plus cobordism $\overline{M_{i+1}-{\rm int}(M_i)}$ is topologically collapsible.
If any of the Questions \ref{plcollapse} or \ref{openplcollapse} has an affirmative answer, then 
$M$ should be homeomorphic to a  PL manifold  
which is  PL collapsible.

\section{CAT(0)  metrics from collapses}\label{constructionmetric}

\subsection{Basic results and techniques} Following \cite{DJ}, it is easy to construct CAT(0) metrics on regular neighborhoods of trees using Gromov's hyperbolization technique. 
To this end, we use metrics along Whitehead's collapsibility (cf.\ \cite{Wh}) as a more direct and suitable (but much less elegant) alternative to Gromov's hyperbolization technique. We recall two critical criteria:

\begin{lem}[cf.\ \cite{DM}]
Consider a locally CAT(K) and locally compact metric length space $X$. 
\begin{compactenum}[(a)]
\item \emph{Cartan--Hadamard theorem.} If $K\le 0$ and $X$ is simply connected, then $X$ is CAT(K).
\item \emph{Bowditch criterion.} If $K> 0$, and every closed curve of length $\le \nicefrac{2\pi}{K}$ can be monotonously contracted to a point, then $X$ is CAT(K).
\end{compactenum}
\end{lem}

\subsection{Links and the Gromov-Alexandrov lemma} 
Recall that the  \emph{star} and \emph{link} of a face $\sigma$ in a 
simplicial complex $\Sigma$ are the subcomplexes 
\[\St_\sigma \Sigma:= \{\theta; \; {\rm there \; exists}\; \tau\in\Sigma \; {\rm such \; that} \; 
\sigma\subseteq\tau\in \Sigma \; {\rm and}\; \theta \subseteq \tau\}\]
\[\Lk_\sigma\Sigma:= \{\theta\in \St_\sigma \Sigma; \;  \theta \cap \sigma =\emptyset\}\]
Although these definitions of stars and links also make sense for any cell complex, 
we wish to emphasize that the combinatorial definition used below in the 
cubical case is slightly different. 
Specifically let $(P, \leqslant)$ be a poset, namely a partially ordered set. For $x\in P$ we denote by 
$\bigwedge x$ the order ideal and by $\bigvee x$ the filter, namely:
\[\bigwedge x =\{ y\in P;  y\leqslant x\}, \; 
\bigvee x =\{y\in P; \; y\geqslant x\}\]
If $x\in P$ we define the {\em link} of the element $x$ in $P$ to be the poset
\[ \Lk_x P = \bigvee x -\{x\}\]
More generally, if $A\subset P$ is a sub-poset, then we set:
\[ \Lk_AP=\bigcap_{x\in A} \Lk_xP\]
A poset  $P$ is called cubical if each order ideal $\bigwedge x$  is the 
product of the poset $I$ associated to the interval 
\[ I=\{0,1,[0,1]; 0 < [0,1], 1 <[0,1]\}\]
Also recall that a poset is called a lattice if every two elements of it 
have a least upper bound and a greatest lower bound. 

To a cell complex $X$ one associates the {\em face poset} $P(X)$, whose elements are
faces (or cells) of $X$ with respect to the inclusion order relation. One usually 
use $\widehat{P(X)}$ to denote the poset obtained by adjoining one minimal element 
and one maximal element to $P(X)$,  elements which can be thought of as $\emptyset$ and $X$ itself.  
With this definition we see that a cubical complex $X$ is a cell complex whose face poset $P(X)$
is a cubical poset and $\widehat{P(X)}$ is  a lattice, namely every couple of elements has a 
{\em least upper bound} and 
a {\em greatest lower bound}.  
We therefore define for an arbitrary cell complex $X$ and $\sigma$ a face of $X$ 
the combinatorial link as being the poset:  
\[ \Lk_{\sigma} X= \Lk_{\sigma} P(X)\] 
In the simplicial case the combinatorial definition matches the geometric one given above, in the sense that the 
face poset associated to the geometric link coincides with the combinatorial link. 
For notation simplicity we will use the same symbol to denote a simplicial complex and 
the corresponding poset.

If $\Sigma$ is a decomposition of a facewise smooth length space, then $\Lk_\sigma \Sigma$ carries a natural facewise spherical length metric. 

\begin{lem}[Gromov--Alexandrov lemma; cf.\ \cite{BH}]
If $\Sigma$ is a locally finite facewise constant curvature $K$ length space. If the link of every face in $\sigma$ in $\Sigma$ has a CAT(1) link, then $\Sigma$ is locally CAT(K).
\end{lem}

The piecewise flat metric on a cube complex is the length metric 
obtained by endowing each cube with a metric making it  
isometric with an unit Euclidean cube.  

In the case of cubical complexes the Gromov criterion from above reads:

\begin{lem}[Gromov--criterion for cubical complexes; cf.\ \cite{BH}]\label{flag}
The piecewise flat metric on a locally finite cube complex 
$\Sigma$ is locally CAT(0) if and only if the link of every vertex of 
$\Sigma$ is flag. 
\end{lem}

A cube complex which has every vertex link a simplicial flag complex.
will be called {\em  nonpositively curved}.

\subsection{Comparison CAT(0) metrics on collapsible complexes}

The following is a slight improvement of Theorem \ref{thm:CAT2}. 

\begin{thm}\label{thm:collcat}
Let $C$ be any finite dimensional collapsible simplicial complex. 
Then there exists a cubical complex $C'$ that is PL homeomorphic to $C$ such that the piecewise flat metric on $C'$ is CAT(0).
 
Let $C=\bigcup_{n=1}^{\infty} C_n$ be the ascending union of 
subcomplexes obtained one from another by (possibly infinitely many) disjoint 
elementary expansions (reverse collapses):  
\[   \{\text{point}\}=C_0\nearrow C_1\nearrow \cdots C_{n-1}\nearrow C_n\nearrow \cdots C\]
Then  $C_{n}$ are convex subsets of $C$, when the later is endowed with the CAT(0) metric 
obtained by pullback from $C'$.
\end{thm}

\begin{proof}
The proof is by a simple induction, constructing the desired facewise hyperbolic CAT(0)-metric along elementary expansions. The induction step is provided by the following result:

\begin{prop}\label{extension}
Let $X$ be a finite simply connected nonpositively curved cube complex. Let $\Gamma$ be a PL $k$-disk which is a subcomplex of $X$. Then there exists: 
\begin{enumerate}
\item a cubical complex $X'$ which contains $X$ as a subcomplex. 
\item a cubical complex $\Delta$ which is a PL $(k+1)$-disk containing the subcomplex 
$\Gamma$ in its boundary $\partial \Delta$
\end{enumerate}
such that  the cubical complex $X'= X\cup_{\Gamma} \Delta$ obtained by gluing  $\Delta$ 
to $X$ along $\Gamma$ is  nonpositively curved. 
\end{prop}

\end{proof}

\subsection{Extensions of nonpositively curved complexes}

Let $X'= X \cup_{\Gamma} \Delta$, where $\Delta = \Gamma\times [0,1]$. 
Our Proposition \ref{extension} would immediately follow if the links of vertices of $\Delta$ within  
$X$ were simplicial flag complexes. 
However, this  might not be true and to achieve it we should allow the cubical complex 
$\Delta$ be suitably modified, while preserving the condition 
that $\Delta$ is a PL $(k+1)$-disk containing $\Gamma$. 

To this purpose, we check the flag condition at each vertex of $\Gamma$.  
If it is not satisfied at some vertex $v$, then we define a sequence of modifications 
eventually producing a new simplicial complex 
$L'_v$  out of the link $\Lk_vX$, such that $L'_v$ is now flag. 
The final step  is to prove that there is some finite 
cube complex $\Delta$ with the property that 
$\Lk_v X'$ is isomorphic to $L'_v$, for any vertex $v$ of $\Gamma$ and 
flag for all the other vertices of $\Delta$. 

In order to construct $\Delta$ out of the links $L'_v$, for $v$ vertex of $\Gamma$, let us introduce more notation. We call an {\em extension} of a simplicial/cubical complex $S$ to be any simplicial/cubical complex containing it. 
We are concerned with the following problem: given a cubical complex $X$ and 
a collection of simplicial complexes $L_{\sigma}, \sigma\in X$, whether 
there exists an extension $X'$ of $X$ such that $\Lk_{\sigma}X'=L_{\sigma}$, for any $\sigma$? 

To this purpose we define a {\em local extension} of $X$ to be 
a collection  $(L_{\sigma}, \sigma\in X)$ of {\em extensions} 
of the links $\Lk_{\sigma}X$ of all proper faces 
$\sigma$ of $X$. This means that the embedding of $\Lk_{\sigma}X$ into $L_{\sigma}$ is part of the data. Moreover, for any 
faces $\tau\subseteq \sigma$ of $X$, we have defined simplicial embeddings maps 
which will be assimilated to inclusions:   
\[ L_{\sigma}\hookrightarrow L_{\tau}, \; {\rm if}\; \tau\subset \sigma.\] 

If $\tau\subset \sigma$ are faces of $X$ and $X'$ is any extension of $X$, 
then $\Lk_{\tau}\sigma$ has a natural embedding in $\Lk_{\tau}X$: 
the poset $\Lk_{\tau}\sigma$ is the poset associated to the face 
of $\Lk_{\tau}X$ corresponding to $\sigma$. 
Therefore we can see $\Lk_{\tau}\sigma$ as a face of the 
simplicial complex $\Lk_{\tau}X$ and in particular as a face of any extension $L_{\tau}$ of 
 $\Lk_{\tau}X$.

\begin{df}
A local extension $(L_{\sigma}, \sigma\in X)$ of a simplicial/cubical complex $X$ is 
said to be {\em coherent} if it commutes with the passage from faces to links, 
namely for every faces $\tau\subset \sigma$ of $X$ the following coherence equation holds: 
\[ \Lk_{\Lk_{\tau}\sigma} (L_{\tau})=L_{\sigma},\]
where $\Lk_{\tau}\sigma$ is seen as a face of $L_{\tau}$. 
\end{df}

Note that $\Lk_{\tau}\sigma\subset \Lk_{\tau}X$, so that its image into $L_{\tau}$ is 
already part of the data. The coherence equation then says that the image of the inclusion 
$L_{\sigma}\hookrightarrow L_{\tau}$ is uniquely determined by the other data.  

We say that a cubical complex $\Delta$ is a   
{\em combinatorial neighborhood} of $\Gamma$ if every $k$-face of $\Delta$ either 
intersects $\Gamma$ or is contained in a unique minimal face intersecting $\Gamma$. 

To any extension $X'=X\cup_{\Gamma}\Delta$ of $X$ we can associate  
a local extension, called its {\em restriction to $X$} consisting of the collection of links 
$L_{\sigma}=\Lk_{\sigma}X'$, along with the natural embeddings 
$L_{\sigma}X\hookrightarrow L_{\sigma}$ and $L_{\sigma}\hookrightarrow L_{\tau}$, when $\tau\subset \sigma$ are faces of $X$. 
The key ingredient in globalizing a local extension is the following lemma: 

\begin{lem}\label{globalizing}
A local extension of a cubical complex $X$ is coherent if and only if it is the 
restriction of a global extension which is a combinatorial neighborhood of a subcomplex of $X$. 
\end{lem}
\begin{proof}
The coherence is satisfied by any extension $X'$ of $X$, since by the definition of the links we have: 
\[ \Lk_{\Lk_{\tau}\sigma} (\Lk_{\tau}X')=\Lk_{\sigma}X'\]

For the nontrivial implication consider 
a coherent local extension $(L_{\tau},\tau\in X)$. 
Define the cubical cells $C_{\alpha}(\tau)$ of dimension $1+\dim\sigma +\dim \alpha$, where 
$\alpha$ is a face of $L_{\tau}$. For fixed $\tau$ these cubes are glued together such that 
\[ C_{\alpha}(\tau)\cap C_{\beta}(\tau)=C_{\alpha\cap \beta}(\tau), \; {\rm for }\; \alpha,\beta\in L_{\tau},\]
where $C_{\emptyset}(\tau)=\tau$. The resulted complex $S(\tau)$ should be 
the closed star of the face $\tau$ in the hypothetical global extension and its 
underlying space is:  
\[ S(\tau)=\bigcup_{\alpha\subset L_{\tau}} C_{\alpha}(\tau)\]

Recall that for $\tau\subseteq \sigma$ we have an inclusion 
$L_{\sigma}\subseteq L_{\tau}$. However the natural map between the stars 
$S(\sigma)\subseteq S(\tau)$ is not the tautological one. 

It is enough to define the embedding $S(\sigma)\subseteq S(\tau)$ by sending 
$C_{\alpha}(\sigma)$ isometrically onto $C_{\alpha*\Lk_{\tau}\sigma}(\sigma)$, where 
$\*$ denotes the join. Note that the image of $\alpha$ in $L_{\tau}$ 
is contained in $\Lk_{\Lk_{\tau}\sigma}L_{\tau}$, by the coherence equations, so that 
indeed $\alpha*\Lk_{\tau}\sigma$ is a face of $L_{\tau}$.

We then define 
\[ X'= \bigcup_{\tau\in X} S(\tau)/\sim\]
where the equivalence relation $\sim$  identifies  $C_{\alpha}(\tau)$ and 
$C_{\alpha*\Lk_{\tau}\sigma}(\sigma)$ if $\tau\subset \sigma$. 
Here we identified $\alpha\in L_{\sigma}$ to its image $\alpha\in L_{\tau}$. 
This complex is well-defined if and only if we have the following commutative diagrams of inclusions: 
\[ \begin{array}{ccc}
S({\sigma}) & \to & S({\tau_1})\\
\downarrow & & \downarrow \\
S({\tau_1}) & \to & S({\theta})\\
\end{array}
\]
for every triple of faces $\theta \subseteq \tau_i \subseteq \sigma$, $i=1,2$. 
But this is a consequence of the fact that, under the above assumptions we have: 
\[ \Lk_{\theta}\tau_1 * \Lk_{\tau_1}\sigma=\Lk_{\theta}\tau_2*\Lk_{\tau_2}\sigma\]
 
Moreover, the coherence of the local extension insures that after gluing we indeed have 
$\Lk_{\sigma}X'=L_{\sigma}$, as desired. 
Moreover, $X'$ is a combinatorial neighborhood of a subcomplex of $X$, 
by construction. 
\end{proof}

To see how a global extension is constructed out of a local extension consider the following 
example which starts from a genuine extension $X\cup_{\Gamma} \Delta$, where $\Delta$ 
is a product, as in the figure below. Let then $L'_{\sigma}$ be the stellar subdivision of 
$\Lk_{\sigma}X'$. 

\begin{figure}[h!tb]
    \includegraphics[width=5.5cm]{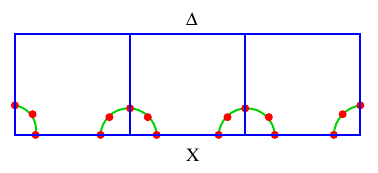}
\caption{A local extension obtained by stellar subdivision of links in an extension}
\end{figure}

Then the collection $(L'_{\sigma}, \sigma\subset X)$ is a coherent local extension of $X$ and 
the global extension associated to it is drawn below:

\begin{figure}[h!tb]
    \includegraphics[width=5.5cm]{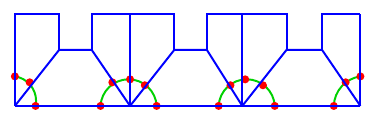}
\caption{The extension obtained from the coherent local extension}
\end{figure}

This construction then can be used for modifying the links of a given extension. Specifically, we have:  

\begin{lem}\label{makeflag}
Assume that the cubical complex $X$ is flag and let $\mathcal Z=(L_{\tau}, \tau\in X)$ be a coherent local extension. 
Then there is another coherent local extension  $\mathcal W=(L'_{\tau}, \tau\in X)$ 
such that:
\begin{enumerate}
\item  every $L'_{\tau}$ is obtained by subdividing faces of $L_{\tau}$  
which are not in $\Lk_{\tau}(X)$;
\item every  link $L'_{\tau}$ is  flag. 
\item the cubical complex $X'$ whose restriction is $\mathcal W=(L'_{\tau}, \tau\in X)$  is a 
combinatorial neighborhood of a subcomplex of $X$. Moreover, every vertex link of $X'$ is also flag.   
\end{enumerate}
\end{lem}
By the previous Lemma \ref{globalizing} it suffices to consider that $\mathcal Z$ is the restriction of 
an extension $X\cup_{\Gamma} \Delta$ by some combinatorial neighborhood of a
cubical subcomplex $\Gamma$. 
For sake of simplicity we can assume that $\Delta=\Gamma\times [0,1]$
which is a combinatorial neighborhood of $\Gamma$.

Let $v$ be a vertex of $\Gamma$ and $\Lk_v X$ be the link of $v$ within $X$. 
As $\Delta$ is a combinatorial neighborhood of $\Gamma$, it follows that 
the link $\Lk_v X$ is the union 
\[C(\Lk_v\Gamma)\cup_{\Lk_v\Gamma} \Lk_v X.\] 
Here $C(\Lk_v\Gamma)$ denotes the cone over $\Lk_v \Gamma$.
Faces of $C(\Lk_v\Gamma)$ containing the cone vertex will be called {\em vertical}. 

If all links   $\Lk_v X$ are flag, we are done. Suppose that there exists some 
link  $\Lk_v X$ which is not flag. This means that there exists at least one 
{\em missing face}, namely a $n$-dimensional 
simplex $\sigma\not\in \Lk_v X$ whose boundary  $\partial \sigma$ is contained 
in  $\Lk_v(X)$. As $\Lk_v X$ is flag, it follows that $\sigma$ has its vertices in the cone 
$C(\Lk_v\Gamma)$, so that at least one face of $\partial \sigma$ is vertical.

We now stellar subdivide all vertical faces of $C(\Lk_v\Gamma)$ in decreasing order of their dimensions, without touching faces of $\Lk_vX$. 

One single subdivision is enough to get rid of the missing face $\sigma$. 
However, if we stop here we might obtain new vertices in the corresponding extension which 
have not flag links. In the picture below we have a vertex $v$ of a square 2-cell of $\Gamma$, 
which is a facet of the 3-cube $\Box$ of $\Delta$. The link $\Lk_v\Box\subset \Lk_v\Delta$   
is a (spherical) 2-simplex $\sigma$. Suppose that we aim at 
stellar subdividing this 2-simplex, as in figure \ref{stellar1}.

\begin{figure}[h!tb]
    \includegraphics[width=5.5cm]{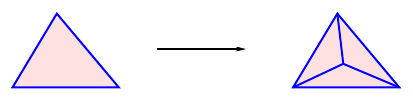}
\caption{One stellar subdivision of a face}\label{stellar1}
\end{figure}
Then the 3-cube $\Box$ from $\St_v\Delta$ is replaced by a 
complex  $Q(\Box,v)$ consisting of the union of three 
3-cubes as in figure \ref{stellar2}.

\begin{figure}[h!tb]
    \includegraphics[width=8.5cm]{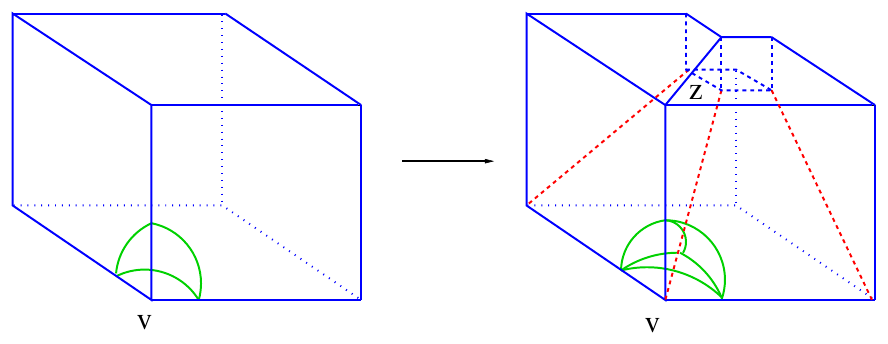}
\caption{One stellar subdivision might create a new vertex of $S(v)$ whose link is not flag}\label{stellar2}
\end{figure}

We can see from the picture that the link of the newly created vertex $z$ of $Q(\Box,v)$ is not flag. 
To remedy that we  have to continue stellar subdividing all cells of $\sigma$ 
which are not cells of $\Lk_vX$, in decreasing order of their dimensions.

Specifically, in the situation above we should continue to subdivide $\sigma$ as indicated 
on figure \ref{stellar3}: 

\begin{figure}[h!tb]
    \includegraphics[width=5.5cm]{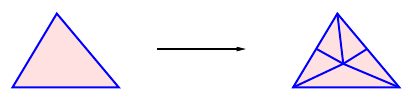}
\caption{Iterated stellar subdivision of a face}\label{stellar3}
\end{figure}

In this picture the horizontal edge of the 2-simplex  $\sigma$ in  $\Lk_v\Delta$ belongs also 
to $\Lk_v X$ and hence it will be left untouched. 
Therefore the 3-cube $\Box$ from $\St_v\Delta$ is now replaced by the 
subcomplex $Q(\Box,v)$ of $S(v)$ consisting of the union of five 3-cubes as in  figure 
\ref{stellar4}.

\begin{figure}[h!tb]
    \includegraphics[width=8.5cm]{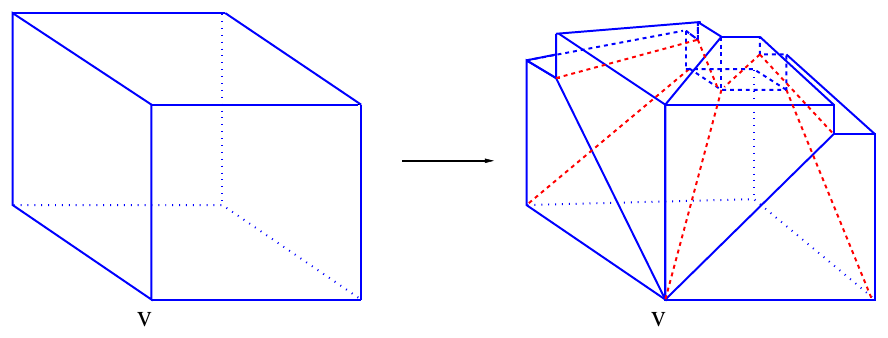}
\caption{Modification of a cube induced by an iterated stellar subdivision}\label{stellar4}
\end{figure}

Let $L'_{\tau}$ be the result of stellar subdividing all vertical cells of $L_{\tau}$, as above.  
As $L_{\tau}$ was coherent  the  collection $(L'_{\tau}, \tau\in X)$ 
is a coherent local extension.  
Furthermore, all links $L'_{\tau}$ are now flag, by construction. 
By Lemma \ref{globalizing} this local extension is the restriction of 
an extension $X'=X\cup_{\Gamma}\Delta$. 
Here $\Delta$ is a combinatorial neighborhood of $\Gamma$, in particular it is a 
PL disk.

One obtains the star $S(v)=\St_vX'$, of $v\in\Gamma$  by  
replacing  the initial star $\St_v \Gamma\times [0,1]$ with the union  
of cubical complexes $Q(\Box, v)$, 
over all top dimensional cell $\Box$ of $\Gamma\times [0,1]$. 
Recall that $Q(\Box,v)$ is simply the union of all cubes $C_{\alpha}(\Box\cap \Gamma)$,  
$\alpha\in L'_{v}\cap \Box$ associated to the iterated stellar subdivision of the simplex 
$\Lk_v\Box$.

We write $\Box^n$ if we want to specify the dimension of $\Box$.  
We will describe $Q(\Box^n,v)$ by induction of the dimension $n$. 
Let $H^n=\Box^n- \frac{1}{4}\Box^n$, where the two cubes are parallel and have 
the vertex $\overline{v}$ opposite to $v$ which is common to both of them.  
Then $Q(\Box^2,v)$ is the subdivision of $H^2$ into two squares, by the diagonal 
connecting $v$ to the vertex $z$ of the smaller cube, which is opposite to $\overline{v}$. 
Assume now that $Q(\Box^{n-1},v)$ is already defined. 
Then $H^n$ has $n$ facets passing through $z$, one of them being horizontal, 
namely parallel to $\Box^n\cap \Gamma$. Note that $H^n$ has also a natural 
intermediary subdivision into 
$n$ cubes, each cube being associated to a facet of $H^n$ passing through $z$ and one
parallel facet of $\Box^n$ passing through $v$.     
Each such facet is isometric to $\frac{1}{4}\Box^{n-1}$ and contains the subcomplex 
$\frac{1}{4}Q(\Box^{n-1},z)$. We consider the $n-1$ prisms: 
\[ T_i^n=\frac{1}{4}\left(\Box^{n-1}-Q(\Box^{n-1},z)\right)\times \left[0 \leq t_i \leq \frac{3}{4}\right]\subset H^n\]
associated to the horizontal coordinates $t_i, 1\leq i\leq n-1$. 
We then see that $Q(\Box^n,v)=H^n-\cup_{i=1}^{n-1} T_i$. 
Every cube from the intermediary subdivision of $H^n$ 
except the one which is associated to the horizontal facets now inherits 
a subdivision in $n$-dimensional cubes obtained by taking the product 
of the subdivision of $Q(\Box^{n-1},v)$ into $(n-1)$-dimensional cubes with an interval.    
There is only one cube from the intermediary subdivision which 
survives as a cube of $Q(\Box^{n},v)$, the horizontal cube determined by the facet 
$\Box\cap \Gamma$ and the horizontal (not subdivided) facet of $\frac{1}{4}\Box^n$.  

The new vertices of $Q(\Box^n,v)$, i.e. those which are off $X$, are of two types: those coming from 
a product of the type $Q(\Box^{n-1},v)\times \left[0,\frac{3}{4}\right]$ and $z$. 
The links of the first type vertices are cones over the corresponding links in $Q(\Box^{n-1},v)$. 
By induction on $n$ they are flag simplicial complexes. 
Eventually, the link of $z$ is just the boundary of a $n$-dimensional simplex without one facet, and hence it is a flag complex again. 

Further we have to glue together $Q(\Box_1^n,v)$ and $Q(\Box_2^n,v)$, where 
$\Box_1$ are two adjacent $n$-cubes containing $v$. The new links 
of common vertices are just unions of the former links within  
$Q(\Box_1^n,v)$ and $Q(\Box_2^n,v)$ respectively, along a common facet which separates and 
hence they remain flag. The star $S(v)$ is the union of complexes $Q(\Box^n,v)$ 
over all cubes of $\Gamma\times[0,1]$ containing $v$. Using induction on their number and 
finally the fact that $\Lk_v\Gamma$ is flag, we obtain that the vertices of 
$S(v)$ have flag links.  

The vertices of $S(v)$ are either of old vertices of $X$, when they belong to $\Gamma$, or 
new vertices, which belong to $X'$ but not to $X$.

From the description given above for $Q(\Box,v)$ we see that 
every new vertex is connected by an edge to an old vertex.

Eventually, to obtain the complex $X'$ we have to identify some cubes of 
$S(v)$ and $S(w)$ if $v$ and $w$ belong to the same cell $\tau$ 
of $\Gamma$. More precisely, we have to identify the cubes 
$C_{\alpha*\Lk_v\tau}(v)$ and $C_{\alpha*\Lk_w\tau}(w)$ and both of them with 
$C_{\alpha}(\tau)$, for any simplex $\alpha\in L'_{\tau}$. 
In our case, this means that the only identifications needed when we glue together 
the stars $S(v)$ are as follows. 
Let $v_{\underline j}$,  for $\underline{j}\in \{0,1\}^{n-1}$ be the vertices of 
some horizontal facet $\Box^n\cap \Gamma$. Then in $X'$ one should identify   
with each other the horizontal $n$-cubes in $S(v_{\underline{j}})$. 
The link of a vertex in $X'$ obtained after identification of two or more cubes 
is then a common subdivision of the former links of the vertices in 
their respective stars $S(v_{\underline{j}})$. Therefore the links of new 
vertices in $X'$ remain flag after gluing and identifying cubes of stars of vertices 
which belong to a cell. 
This implies that all vertices of the resulting cube complex $X'$ have flag links.

\begin{proof}[Proof of Proposition \ref{extension}] 
The result is a direct consequence of Lemmas \ref{globalizing}, \ref{makeflag} and 
Gromov-Alexandrov's Lemma \ref{flag}. 
\end{proof}

\subsection{Proof of Theorem \ref{thm:collcat}}  
Consider a sequence of  reverse collapses $C_{n-1}\nearrow C_{n}$, $C_0=\{point\}$, such that 
$C=\bigcup_{n=1}^{\infty} C_n$. 
Assume that we constructed  CAT(0) cubical complexes $C'_i$, $0\leq i \leq n$
such that $C'_i$ are PL homeomorphic to $C_i$ and moreover $C'_{j}$ are convex in $C'_i$, 
for $j\leq i$.  

Consider the next  elementary expansion $C_{n} \nearrow C_{n+1}$, where $\delta$ 
denotes the $k$-cell attached and set $\gamma=\delta \cap  C_{n}$.  
As $\gamma$ is a PL $(k-1)$-disk, its image $\Gamma$ within $C'_n$ by the 
PL homeomorphism $C_n\to C'_n$ is also a PL $(k-1)$-disk. 
By Proposition \ref{extension} there is a PL $k$-disk $\Delta$ 
containing $\Gamma$ embedded within $\partial \Delta$   
such that $C'_{n+1}=C'_n\cup_{\Gamma} \Delta$ is a CAT(0) cubical complex.  
It then follows that the PL homeomorphism $C'_n\to C_n$ extends 
to $C'_{n+1}\to C_{n+1}$. Then the claim follows by induction on $n$. 

To make sure that $C'_{n}$ is convex within $C'_{n+1}$, we first note that 
any cubical cell of $C'_{n+1}$ intersects $C'_n$ along a face. 
Therefore $C'_n$ is locally convex, with respect to the piecewise flat metric
on $C'_{n+1}$. Moreover $C'_n$ is connected and it is well-known that 
a connected simply connected  locally convex  subcomplex of a 
CAT(0) cubical complex is convex, e.g. as a consequence of 
(\cite{BH}, Prop II.4.41).  
Eventually $C_n$ is convex in $C$ as well, as the CAT(0) metric on $C_n$ won't 
be changed when we adjoin new cells.

\begin{rem}
When $C$ is infinite, one might allow, instead of a single 
elementary expansion to step from $C_n$ to $C_{n+1}$,  
infinitely many simultaneous elementary expansions 
having disjoint attaching maps. We obtain the 
convexity of $C_n$ within $C$ by the same arguments. 
\end{rem}

\section{Variations}\label{variations}
\subsection{More tameness conditions}

\begin{df}
An open manifold $M$ is {\em weakly geometrically $k$-connected} (see \cite{FG}) if $M=\cup_{j=1}^{\infty} K_j$, where 
$K_j\subset {\rm int}(K_{j+1})$, for $j\geq 1$, is an exhaustion by compact $k$-connected PL manifolds. 
When $k=\infty$ we use the term  weak geometric contractibility.  
\end{df}

It is obvious that CAT(0) polyhedra are  weakly geometrically contractible. It suffices to consider any exhaustion by metric balls, which are convex. To guarantee the filtration is a filtration by manifolds, one merely has to to pass to the regular neighborhoods of these geometric balls to obtain the desired filtration.

\begin{df}
An end is of type $F_k$ (respectively $F$) if it admits arbitrarily small clean neighborhoods with the homotopy type of a CW complex having finite $k$-skeleton (respectively finitely many cells).  
\end{df}
This generalizes the Tucker condition explored in \cite{MT} which requires  that 
the complement of any compact subpolyhedron has finitely generated fundamental group, i.e.\ is of type $F_1$.


 \subsection{Weak geometric contractibility is not sufficient}
The aim of this section is to construct examples of 
open weakly geometrically contractible manifolds which are neither semistable nor with end of type $F_1$. 

\begin{df}
An open manifold $W$ has {\em injective} ends if it admits 
an ascending compact exhaustion by submanifolds $K_j$ 
with the property that the maps induced by inclusions 
$\pi_1(\partial_* K_j)\to \pi_1(K_{j+1}-{\rm int}(K_j))$
and $\pi_1(\partial_* K_{j+1})\to \pi_1(K_{j+1}-{\rm int}(K_j))$ 
are injective. Here $\partial_*K$ denotes an arbitrary 
connected component of $\partial K$. 
The ends of $W$ are {\em strictly injective} if none of the maps 
above are surjective. 
\end{df}

It is well-known (see e.g. \cite{Geogh}, \cite{Guilb13}, ex.4.17) that: 
\begin{lem}
An open manifold with strictly injective ends is not semistable. 
\end{lem}

Our goal now is to construct geometrically contractible manifolds 
with strictly injective ends. To this purpose we introduce more terminology. 
 
 We say that the nontrivial pair $H\subset G$ of finitely presented groups  is {\em tight} if the normal closure 
 of $H$ within $G$ is $H$ itself, i.e.\ there is no proper normal subgroup of $G$ containing $H$.
 The pair is nontrivial if $H\subset G$ is proper.  
The group $G$ is {\em superperfect} if $H_1(G)=H_2(G)=0$. Observe that  $G$ is superperfect 
if and only if  $G=\pi_1(K)$, where $K$ is a finite complex 
whose integral homology is that of a point. 

\begin{lem}
Given a  nontrivial tight pair $H\subset G$ of superperfect finitely presented groups, there exists an open 
geometrically contractible manifold $W$ with a contractible compact exhaustion $K_j$ such that  the maps
$\pi_1(\partial_* K_j)\to \pi_1(K_{j+1}-{\rm int}(K_j))$
and $\pi_1(\partial_* K_{j+1})\to \pi_1(K_{j+1}-{\rm int}(K_j))$ 
are given by the proper inclusions $H\subset G$.  
\end{lem}
\begin{proof}
A classical result of Kervaire (\cite{Kerv}) states that $G$ is the fundamental group of a 
homology sphere $\Sigma^n$ of dimension $n\geq 5$ if (and only) if $G$ is finitely presented and superperfect. 
Let $H=\pi_1(K)$ be fundamental group of an acyclic $k$-complex, $k\geq 2$. 
Choose $n\geq 2k+1$, in order to be able to embed $K\to \Sigma^n$ such that 
the map induced by inclusion  $\pi_1(K)\to \pi_1(\Sigma^n)$ corresponds to the inclusion 
$H\hookrightarrow G$. Consider two such embeddings $K_1$ and $K_2$, which by transversality could be assumed to be 
disjoint. Let $N_1$ and $N_2$ denote disjoint regular neighborhoods of $K_1$ and $K_2$ within $\Sigma^n$. 

Using general position we derive that $\pi_1(\Sigma-{\rm int}(N_1\sqcup N_2))\cong \pi_1(\Sigma)=G$ and 
$\pi_1(\partial N_i)\cong\pi_1(N_i-K_i)\cong\pi_1(N_i)=H$. 
Moreover, the map $\pi_1(\partial N_i)\to \pi_1(\Sigma)$ induced by the inclusion is identified with the  
embedding $H\hookrightarrow G$. 
If $H$ is weakly acyclic then $X=\Sigma-{\rm int}(N_1\sqcup N_2)$ has the homology of a spherical cylinder. 

Now, since $\partial N_1$ is a homology sphere of dimension at least 4, it bounds a 
compact contractible manifold $M$. 
Then, the result of gluing $M\cup X$ is acyclic and simply connected and hence contractible. 
By recurrence we find that $K_j=M\cup X\cup X\cdots \cup X$, where $X$ occurs $j$-times, is also contractible. 
Therefore the open manifold $W=M\cup X\cup X\cdots$ is geometrically contractible and 
the exhaustion $K_j$ satisfies all the requirements. 
\end{proof}

\begin{lem}
Any finite superperfect group $H$ is contained in a superperfect group $G$ to form a nontrivial tight pair. 
In particular, this is the case for 
the binary icosahedral group  $\langle a, b | a^5=b^3=(ab)^2\rangle\cong SL_2(\mathbb F_5)$, 
$SL(2,\mathbb F_p)$, for odd prime $p$, or more generally 
any finite perfect balanced group.  
\end{lem}
\begin{proof}
Any finite group is contained into some $S_n$ which is contained 
into $Sp(2n, \mathbb F_q)$. 
The finite symplectic group $Sp(2n, \mathbb F_2)$ is simple (hence perfect) for $n\geq 4$ and has trivial Schur multiplier so that it is superperfect. 
The finite symplectic groups $PSp(2n,\mathbb F_q)$ are simple for $n\geq 4$ 
and have Schur multiplier $\Z/2\Z$, when $q$ is odd, so that 
 $Sp(2n, \mathbb F_q)$ is the universal central extension of 
  $PSp(2n, \mathbb F_q)$. Therefore it is superperfect. 
Any proper normal subgroup of  $Sp(2n, \mathbb F_q)$ should be 
contained in the center, so that the pair obtained is tight and nontrivial. 

Note that $SL_2(\mathbb F_p)$, for odd prime $p$ are perfect and admit balanced presentations (see \cite{CR}). 
Thus their presentation 2-complexes are acyclic since their Schur multiplier is trivial, by an old theorem of Schur. 

Alternatively any finite group is contained in the Thompson group $V$, which is finitely presented, simple 
and  superperfect (\cite{Kapou}).
Moreover $V$ is non co-Hopfian, as it contains copies of $V$ corresponding to stabilizers 
of dyadic intervals of the circle, when we identify $V$ with a group of piecewise linear dyadic bijections 
of the circle. Therefore these inclusions $V\hookrightarrow V$ are  nontrivial tight pairs of infinite groups. 
 \end{proof}

\begin{rem}
More examples of  superperfect groups are 1-relator torsion-free groups (Lyndon's theorem) and 
perfect finitely presented groups of deficiency zero, 
in particular Higman's groups, whose presentation complexes are acyclic. 
Other finite examples are $SL_2(\mathbb F_{8})$, $SL_2(\mathbb F_{32})$, $SL_2(\mathbb F_{64})$, 
$SL_2(\mathbb F_{27})$, $SL_2(\mathbb F_{5})\times SL_2(\mathbb F_{5})$, 
$\widehat{A_7}$, etc (see \cite{CRKW}). 
More recently the Burger-Mozes examples (see \cite{BMo}) 
of simple finitely presented torsion-free groups acting on products of trees can be written 
as amalgamated product of free groups over free subgroups. A classical theorem of Whitehead 
(\cite{Wh}) states that whenever we have aspherical spaces $X$ and $Y$ such that 
$\pi_1(X\cap Y) \to \pi_1(X)$ and $\pi_1(X\cap Y)\to \pi_1(Y)$ are injective, then 
$X\cup Y$ is aspherical. This proves that Burger-Mozes simple groups are superperfect. 
\end{rem}

\bibliographystyle{plain}

\end{document}